\pdfoutput=1
\documentclass{amsart}
\usepackage{amsmath,amsfonts,amsthm,amssymb,nameref,mathtools}
\usepackage[top=3cm, bottom=3cm, left=3cm, right=3cm]{geometry} 
\usepackage{enumerate}
\usepackage{verbatim}
\usepackage{pinlabel}
\usepackage{graphicx}
\usepackage{caption}
\usepackage{subcaption}
\usepackage[pdftex]{hyperref}
\usepackage[all]{hypcap} 
\usepackage{microtype} 
\hypersetup{
    pdftitle={Morse quasi-geodesics, superlinear divergence, and sublinear contraction},    
    pdfauthor={Arzhantseva, Cashen, Gruber, Hume},     
    pdfkeywords={Morse quasi-geodesic, contracting projection, superlinear
  divergence, geodesic image theorem}, 
    colorlinks=true,       
    linkcolor=black,          
    citecolor=black,        
    filecolor=black,      
    urlcolor=black           
}

\title[Morse quasi-geodesics via divergence and contraction]{Characterizations of Morse quasi-geodesics
via superlinear divergence and sublinear contraction}

\author{Goulnara N.\ Arzhantseva}
\address{Faculty of Mathematics\\ University of Vienna\\ Oskar-Morgenstern-Platz~1\\ 1090 Vienna, Austria}
 \email{\href{mailto:goulnara.arzhantseva@univie.ac.at}{\texttt{goulnara.arzhantseva@univie.ac.at}}}

\author{Christopher H.\ Cashen}
\address{Faculty of Mathematics\\ University of Vienna\\ Oskar-Morgenstern-Platz~1\\ 1090 Vienna, Austria}
  \email{\href{mailto:christopher.cashen@univie.ac.at}{\texttt{christopher.cashen@univie.ac.at}}}

\author{Dominik Gruber}
\address{Institut de Math\'ematiques\\ Universit\'e de Neuch\^atel\\
  Rue Emile-Argand~11\\ 2000 Neuch\^atel, Switzerland}
\curraddr{Department of Mathematics\\
ETH Z\"urich\\
R\"amistrasse 101\\
8092 Z\"urich, Switzerland}
 \email{\href{mailto: dominik.gruber@math.ethz.ch}{\texttt{dominik.gruber@math.ethz.ch}}}

\author{David Hume}
\address{Mathematical Institute\\ University of Oxford\\ Woodstock Road\\ Oxford Ox2 6GG, United~Kingdom}
\email{\href{mailto: david.hume@maths.ox.ac.uk}{\texttt{david.hume@maths.ox.ac.uk}}}

\subjclass[2010]{Primary: 20F65; Secondary: 20F67}
\keywords{Morse quasi-geodesic, contracting projection, superlinear divergence, geodesic image theorem}

\newcommand{\N}	{\mathbb N}
\newcommand{\Z}	{\mathbb Z}
\newcommand{\R}	{\mathbb R}

\newcommand{\Cay}	{\operatorname{Cay}}
\newcommand{\diam}	{\operatorname{diam}} 

\mathtoolsset{centercolon} 
\newcommand{\from}{\colon\thinspace}

\renewcommand{\setminus}{\smallsetminus}

\newcommand{\almost}{\epsilon}
\newcommand{\asympleq}{\preceq}
\newcommand{\asympgeq}{\succeq}

\newtheorem{theorem}{Theorem}[section] 
 
\newtheorem{proposition}{Proposition}[section] 
\newtheorem{lem}{Lemma} [section] 
\newtheorem{lemma}{Lemma}[section] 
\newtheorem{cor}{Corollary}[section] 
\newtheorem{corollary}{Corollary}[section] 
\newtheorem*{thm*} {Theorem} 
\newtheorem*{prop*}{Proposition}
\newtheorem*{lem*} {Lemma} 
\newtheorem*{cor*} {Corollary}
\theoremstyle{definition}

\newtheorem{defi}{Definition}[section] 
\newtheorem{definition}{Definition}[section] 
\newtheorem{example}{Example}[section]

\newtheorem*{defi*}{Definition}
\newtheorem*{example*}{Example}
\newtheorem*{remark*}{Remark}
\newtheorem*{problem*}{Problem}
\newtheorem*{convention*}{Convention}
\newtheorem*{defi1*}{Definition I of the WPD element}
\newtheorem*{defi2*}{Definition II of the WPD element}

\makeatletter 
\let\c@question=\c@thm 
\makeatother
\makeatletter 
\let\c@theorem=\c@thm 
\makeatother
\makeatletter 
\let\c@proposition=\c@thm 
\makeatother
\makeatletter 
\let\c@corollary=\c@thm 
\makeatother
\makeatletter 
\let\c@definition=\c@thm 
\makeatother
\makeatletter 
\let\c@lemma=\c@thm 
\makeatother
\makeatletter 
\let\c@lem=\c@thm 
\makeatother
\makeatletter 
\let\c@prop=\c@thm 
\makeatother
\makeatletter 
\let\c@cor=\c@thm 
\makeatother
\makeatletter 
\let\c@defi=\c@thm 
\makeatother
\makeatletter 
\let\c@example=\c@thm 
\makeatother
\makeatletter 
\let\c@remark=\c@thm 
\makeatother
\def\makeautorefname#1#2{\expandafter\def\csname#1autorefname\endcsname{#2}}
\let\fullref\autoref

\makeautorefname{thm}{Theorem} 
\makeautorefname{lem}{Lemma} 
\makeautorefname{prop}{Proposition} 
\makeautorefname{cor}{Corollary} 
\makeautorefname{defi}{Definition}
\makeautorefname{theorem}{Theorem} 
\makeautorefname{lemma}{Lemma} 
\makeautorefname{proposition}{Proposition} 
\makeautorefname{corollary}{Corollary} 
\makeautorefname{definition}{Definition}
\makeautorefname{remark}{Remark}
\makeautorefname{example}{Example}
\makeautorefname{section}{Section}
\makeautorefname{subsection}{Section}
\makeautorefname{subsubsection}{Section}
\makeautorefname{question}{Question}

\begin{document}
\begin{abstract}
We introduce and begin a systematic study of sublinearly contracting
projections.

We give two characterizations of Morse quasi-geodesics in an arbitrary geodesic
metric space. One is that they are sublinearly contracting; the other
is that they have completely superlinear divergence.

We give a further characterization of sublinearly contracting projections in terms
of projections of geodesic segments.
\end{abstract}
\maketitle

\section{Introduction}
This paper initiates a systematic study of \emph{contracting projections}.
The aim is to clarify and quantify ways in which a subspace of a geodesic metric space can `behave like' a convex subspace of a hyperbolic space. 

The definition of hyperbolicity captures the notion that a space is
uniformly negatively curved on all sufficiently large scales. 
Following Gromov's seminal paper \cite{Gro87}, hyperbolic groups and spaces have been intensively studied and many generalizations of this notion have been considered.

One particular collection of ideas focus on finding `hyperbolic directions', geodesics that have some of the features exhibited by geodesics in hyperbolic spaces, for instance, those that satisfy the \emph{Morse lemma}, have \emph{superlinear divergence} or satisfy some \emph{contraction} hypothesis.
These ideas find application to Mostow rigidity in rank 1 \cite{Pau96}, the Rank Rigidity Conjecture for
CAT(0) spaces \cite{BalBri95,BesFuj09,CapFuj10}, and hyperbolicity of the curve complex of a hyperbolic
surface \cite{Min96,MasMin99}.
Recently, the concept of \emph{strongly contracting projection} has
been a topic of intense interest in relation to mapping class groups
and outer automorphisms of free groups \cite{Alg11,BesBroFuj15},
acylindrically hyperbolic groups \cite{DahGuiOsi11,Osi16}, and contracting/Morse boundaries \cite{Sul12,Sul14,ChaSul15,Cor15,Mur15}.

We introduce a more general version of contracting projection than has
been previously studied.
Our main result is that this new version of contraction is equivalent
to the Morse property and to a certain superlinear divergence
property.
We give
quantitative links between these various geometric properties. 
We also generalize several fundamental theorems about stronger versions
of contraction to our new, more general, context.

In this paper we establish fundamental results in a very
general setting, so that they will be broadly applicable. 
Indeed, the novel version of contracting projections we introduce here is
essential in a subsequent
paper \cite{ArzCasGrua}, in which  we explore the geometry of
finitely generated graphical small cancellation groups, a class that
includes the Gromov monster groups as notorious examples.
In that paper we engineer finitely generated groups with Cayley graphs
that mimic the surprising geometry of our examples from \fullref{sect:examples}.
In particular, the new spectrum of contracting behaviors in geodesic
metric spaces that we 
discover here does appear in the setting of Cayley graphs of
finitely generated groups.
We also, in \cite{ArzCasGrua}, use the equivalence between sublinear contraction and the Morse
property established here in \fullref{theorem:morse_equiv_contracting} to
characterize Morse geodesics in certain families of graphical small
cancellation groups.

Since the preprint version of this article appeared there have already
been other applications of our results, including work of Cordes and
Hume \cite{CorHum16} and Cashen and Mackay \cite{CasMac17} on Morse boundaries of finitely
generated groups and work of Aougab, Durham, and Taylor
\cite{AouDurTay16} on cocompact subgroups of mapping class groups and $\mathrm{Out}(F_n)$.

\medskip
We give detailed introductions to the three main geometric properties
in Sections \ref{sec:introcontracting}, \ref{intro:Morse},
and \ref{intro:divergence} and  make  precise statements of our
results in Sections \ref{intro:maintheorems} and \ref{intro:further}.

\subsection*{Acknowledgements}
The authors thank the referee for the careful reading and helpful
remarks, and the Isaac Newton Institute for 
  Mathematical Sciences for support and hospitality during the 
  program 
\textit{Non-positive curvature: group actions and cohomology} 
where  work on this paper was undertaken. 
This work is partially supported 
  by EPSRC Grant Number EP/K032208/1.

Goulnara N.\ Arzhantseva was partially supported by the ERC grant ``ANALYTIC'' 
  no.\ 259527 and a grant from the Simons Foundation.
Christopher H.\ Cashen was supported by the Austrian Science Fund
  (FWF):M1717-N25.
Dominik Gruber was supported by the 
Swiss National Science Foundation 
Professorship FN  PP00P2-144681/1.
David Hume was supported by the 
ERC grant no.\ 278469 
and the grant ANR-14-CE25-0004 ``GAMME''.

\subsection{Contracting projections}\label{sec:introcontracting}
Let $Y$ be a subspace of a geodesic metric space $X$, and let
$\almost\geqslant 0$.
The \emph{$\almost$--closest point projection of $X$ to $Y$} is
 the map $\pi^\almost_Y\from X \to 2^Y$ sending a point $x\in X$ to
the set:
\[
 \pi^\almost_Y(x) := \{y\in Y \mid d(x,y)\leqslant d(x,Y)+\almost\}\subset Y
\]
We do not assume the sets $\pi^\almost_Y(x)$ have uniformly bounded diameter.
Note that given any $x\in X$, $\emptyset\neq Y\subset X$, and $\almost>0$, the set
$\pi^\almost_Y(x)$ is non-empty.

\begin{definition}
The $\almost$--closest point projection $\pi_Y^\almost\from X\to
2^Y$ is \emph{$(\rho_1,\rho_2)$--contracting} if the following
conditions are satisfied.
\begin{itemize}
\item The empty set is not in the image of $\pi_Y^\almost$.
\item The functions\footnote{The term `function' always refers to a
    real valued function whose domain, unless otherwise noted, is the
    non-negative real numbers.} $\rho_1$ and $\rho_2$ are non-decreasing and
  eventually non-negative.
\item The function $\rho_1$ is unbounded and $\rho_1(r)\leqslant r$.
\item For all $x,\, x'\in X$, if $d(x,x')\leqslant \rho_1(d(x,Y))$ then:
  \[\diam \pi_Y^\almost(x)\cup\pi_Y^\almost(x')\leqslant\rho_2(d(x,Y))\]
\item $\lim_{r\to\infty}\frac{\rho_2(r)}{\rho_1(r)}=0$.
\end{itemize}
\end{definition}

We say that \emph{$Y$ is $(\rho_1,\rho_2)$--contracting} if there exists
$\almost\geqslant 0$ such that $\pi_Y^\almost$
is $(\rho_1,\rho_2)$--contracting, see \fullref{def:spacecontracting}.
We say a collection of subspaces $\{Y_i\}_{i\in\mathcal{I}}$ is
\emph{uniformly contracting} if there exist $\rho_1$ and $\rho_2$ such
that for all $i\in\mathcal{I}$, the subspace $Y_i$ is $(\rho_1,\rho_2)$--contracting.

The rough idea is that, asymptotically as $x$ gets far from $Y$, if
$B$ is a ball centered at $x$ and disjoint from $Y$
then the diameter of its projection is negligible compared to the
diameter of $B$.
More accurately, this is true at a specific scale --- when the radius of $B$ is
$\rho_1(d(x,Y))$.
We claim no finer control of the projection diameter when $B$ has
smaller radius. 

For a simple, but conceptually useful, example, consider a circle $X$
and an arc $Y\subset X$. 
Take $\rho_1(r):=r$, and let $\rho_2$ be the constant function whose
value is the distance between the endpoints of $Y$.
Then $\pi_Y^0$ is $(\rho_1,\rho_2)$--contracting.
There is a unique point $x\in X$ farthest from $Y$. 
The ball $B$ of radius $\rho_1(d(x,Y))$ about $x$ is all of
$X\setminus Y$, and $\pi_Y^0(B)=\pi_Y^0(x)$ has diameter
$\rho_2(d(x,Y))$.

The simplest example that is not $(\rho_1,\rho_2)$--contracting
for any choice of $\rho_1$ and $\rho_2$ is to take $X$ to be the
Euclidean plane and take $Y$ to be a geodesic. 
Then the diameter of $\pi_Y^0$ of any ball is equal to the diameter of
the ball, so we cannot satisfy $\lim_{r\to\infty}\frac{\rho_2(r)}{\rho_1(r)}=0$.

The simplest contracting example with $Y$ unbounded is to take $X$ to be a tree
and $Y$ to be an unbounded convex subspace. 
Then $\diam \pi_Y^0(B_{d(x,Y)}(x))=0$ for
every $x$, so $\pi_Y^0$ is $(r,0)$--contracting. 
In more general
$\delta$--hyperbolic spaces, $\almost$--closest point projection to a geodesic is $(r,D)$--contracting
for some $D$ depending only on $\delta$ and $\almost$. 
Such a case, when $\rho_1(r):=r$ and $\rho_2$ is bounded, is called
\emph{strongly contracting}.

Pseudo-Anosov axes in Teichm\"uller space are strongly contracting
\cite{Min96}, as are iwip axes in the Outer Space of the outer
automorphism group of a free group \cite{Alg11} and axes of rank 1
isometries of CAT(0) spaces \cite{BalBri95,BesFuj09}. 

We say that $\pi_Y^\almost$ is \emph{semi-strongly contracting} if it 
is $(\rho_1,\rho_2)$--contracting for $\rho_1(r):=r/2$ and $\rho_2$
bounded. 
Related notions have been considered in the context of the mapping class group of a hyperbolic surface \cite{MasMin99, Beh06,DucRaf09}.

We say that $\pi_Y^\almost$ is \emph{sublinearly contracting} if it 
is $(\rho_1,\rho_2)$--contracting for $\rho_1(r):=r$.
In this case the definition implies $\rho_2$ is a sublinear function,
see \fullref{def:sublinear}.
Similarly, $\pi_Y^\almost$ is \emph{logarithmically contracting} if
it is $(\rho_1,\rho_2)$--contracting for $\rho_1(r):=r$ and $\rho_2$ logarithmic.

A schematic diagram of different contracting behaviors is given in \fullref{fig:typesofcontraction}.
A wide range of examples are presented in \fullref{sect:examples}.

\begin{figure}[h]
\captionsetup[subfigure]{labelformat=empty}
  \centering
  \begin{subfigure}[h]{.32\textwidth}
    \centering
\includegraphics[width=.5\textwidth]{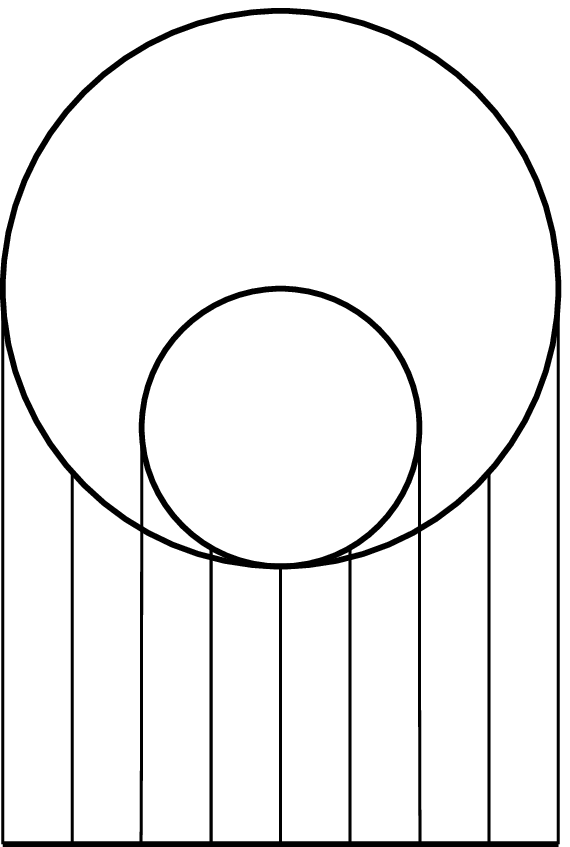}
\caption{Not contracting}
  \end{subfigure}
\hfill
\begin{subfigure}[h]{.32\textwidth}
    \centering
\includegraphics[width=.5\textwidth]{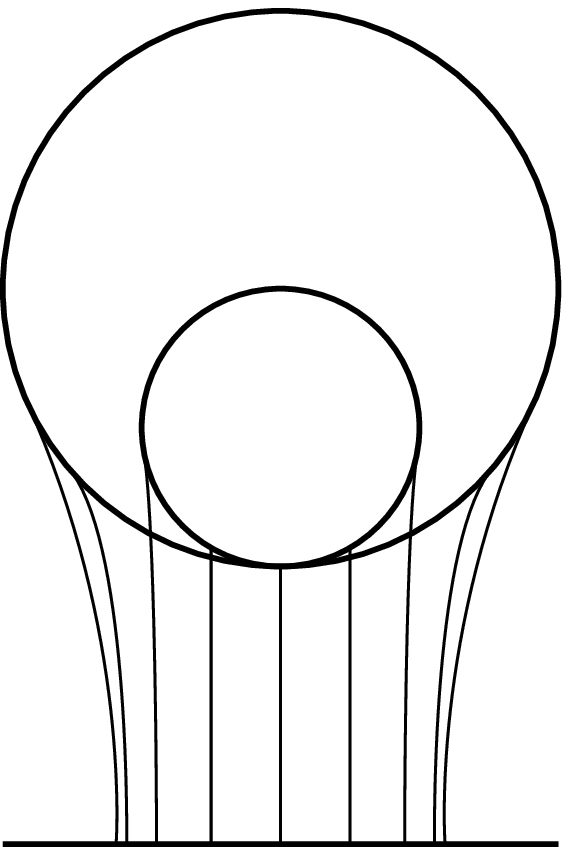}
\caption{Sublinearly contracting}
  \end{subfigure}
\hfill
\begin{subfigure}[h]{.32\textwidth}
    \centering
\includegraphics[width=.5\textwidth]{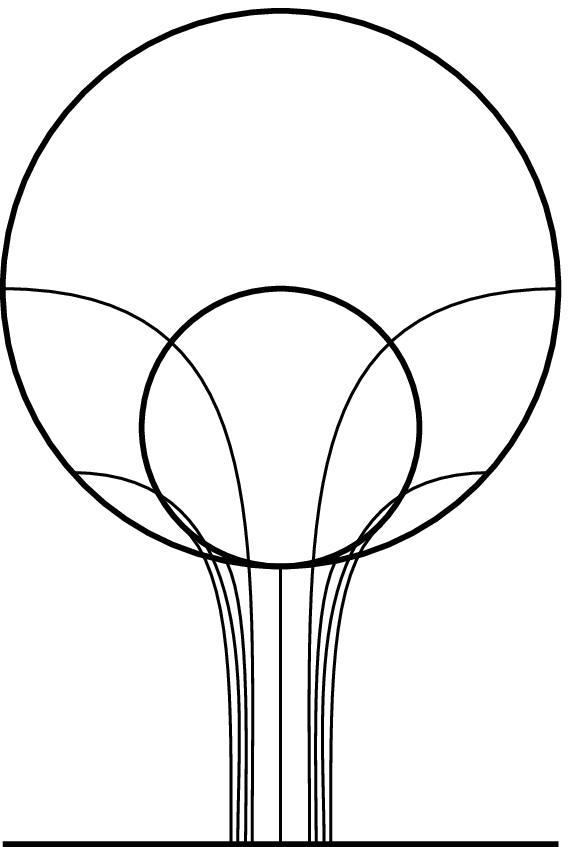}
\caption{Strongly contracting}
  \end{subfigure}
  \caption{Types of contraction}
  \label{fig:typesofcontraction}
\end{figure}

\subsection{The Morse property}\label{intro:Morse}
\begin{definition}
A subspace $Y$ of a geodesic metric space $X$ is \emph{$\mu$--Morse}
for a function $\mu\from [1,\infty)\times [0,\infty)\to [0,\infty)$ if
  for every $L\geqslant 1$ and $A\geqslant 0$, every
  $(L,A)$--quasi-geodesic $\gamma$ with endpoints on $Y$ remains
  within distance $\mu(L,A)$ of $Y$.

  The subspace $Y$ is called \emph{Morse}, or is said to \emph{have the
  Morse property}, if it $\mu$--Morse for some function $\mu$.
A collection of subspaces $\{Y_i\}_{i\in \mathcal{I}}$ is said to be \emph{uniformly
  Morse} if there exists a function $\mu$ such that for all
$i\in\mathcal{I}$ the subspace $Y_i$ is $\mu$--Morse.
\end{definition}

Morse quasi-geodesics have been intensively studied: 
they play a key role in boundary theory for hyperbolic and relatively
hyperbolic groups.
Recently, the Charney school \cite{Sul12,Sul14,ChaSul15,Cor15,Mur15} has been generalizing such boundary
theories to arbitrary proper geodesic metric spaces using the so called
`Morse boundary' consisting of asymptotic equivalence classes of Morse rays.

 Morse quasi-geodesics have been characterized\footnote{See also the
   related ``middle recurrence'' characterization of the Morse property
   in \cite{DruMozSap10}.} in terms of cut-points
 in asymptotic cones \cite{DruMozSap10}: a quasi-geodesic $q$ in $X$
 is Morse if and only if every point $x$ in the limit $\mathfrak{q}$ of $q$ in any asymptotic cone $\mathcal{C}$ of $X$ is a cut-point separating ends of $\mathfrak{q}$;
 that is, $\mathcal{C}\setminus\{x\}$ has at least two connected components containing points of $\mathfrak{q}$.
Cut-points in asymptotic cones are a key element of the proof of the quasi-isometry invariance of relatively hyperbolic (asymptotically tree-graded) spaces \cite{DruSap05}. 
It remains a very important open question to determine whether a space
in which every asymptotic cone admits a cut-point necessarily admits a
Morse quasi-geodesic. 

As a result, it is of great interest to find and classify Morse
quasi-geodesics.
If a solvable group admits a Morse quasi-geodesic then it is virtually 
cyclic, and the same holds for any other group satisfying a 
non-trivial law, for instance, a torsion group with bounded exponent
\cite{DruSap05}.
At the other extreme, every quasi-geodesic in a hyperbolic space is Morse.
There are non-trivial classifications of Morse quasi-geodesics for relatively hyperbolic groups
\cite{Osi06} and CAT($0$) spaces \cite{BalBri95,BesFuj09,Sul14}.
We use the tools of this paper to perform such a classification for
 graphical small cancellation groups in 
\cite{ArzCasGrua}.

\subsection{Divergence}\label{intro:divergence}
Closely related to the study of Morse quasi-geodesics is the notion of
\emph{divergence}.
The definition is technical, so we postpone it until \fullref{def:divergence}.
The idea is that the divergence of a quasi-geodesic $\gamma$ in a
space $X$ is a
function whose value at $r$ is the minimal length of a path in $X$
circumventing a ball of radius $r$ centered on $\gamma$.
In our version of divergence we allow the forbidden ball to be centered
at different points of $\gamma$ for different values of $r$. 
Some authors require the balls to have fixed center at $\gamma(0)$.

Morse geodesics were used to produce cut points in asymptotic cones. 
Divergence can be used to rule them out \cite{DruMozSap10}: if $G$ is a finitely
generated group then no asymptotic cone of $G$ admits a cut point if
and only if there exists a constant $K$ such that for any finite
geodesic $[a,b]$ with midpoint $c$, there is a path from $a$ to $b$
avoiding the ball centered at $c$ with radius $d(a,b)/4 - 2$ of length at most
$Kd(a,b)+K$. 
The interplay between divergence and Morse quasi-geodesics is explored in \cite{DruMozSap10} and \cite{BehDru14}.

Morally, for a quasi-geodesic $\gamma$ the Morse property and linear divergence
are opposites. 
The Morse property says good (quasi-geodesic) paths between points of $\gamma$ stay close to
$\gamma$, and linear divergence says it is easy for a path
between points of $\gamma$ to stray far from $\gamma$.
However, there are some subtleties.
There are groups that admit quasi-geodesics with superlinear
divergence, yet have an asymptotic cone with no cut point, and
therefore no Morse quasi-geodesics \cite{OlsOsiSap09}.
By construction, for each of these groups there is
an unbounded sequence $(r_n)$ such that the divergence is linear (it
satisfies the above conditions for a fixed $K$) whenever
$d(a,b)=r_n$ for some $n$.
We say a geodesic metric space has
\emph{completely superlinear divergence} if no such unbounded
sequence exists.
We show in \fullref{theorem:morse_equiv_csld} that this is the precise divergence property that characterizes
Morse quasi-geodesics.

\subsection{Main theorems}\label{intro:maintheorems}
Restricted to quasi-geodesics, our main results say:
\begin{theorem}
  Let $X$ be a geodesic metric space. Let $\gamma$ be a quasi-geodesic
  in $X$. 
The following are equivalent:
\begin{enumerate}
\item $\gamma$ is sublinearly contracting.
\item $\gamma$ is Morse.
\item $\gamma$ has completely superlinear divergence.
\end{enumerate}
\end{theorem}
Special cases of this theorem have appeared before. 
If $X$ is hyperbolic then these conditions are well-known properties
of arbitrary quasi-geodesics, and conditions (1) and (3) can be
strengthened to `strongly contracting' and `at least exponential divergence',
respectively. 
If $X$ is CAT(0) and $\gamma$ is a geodesic then this is a recent theorem of
Charney and Sultan \cite{ChaSul15}. 
In that case, conditions (1) and (3) can be strengthened to `strongly
contracting' and `at least quadratic divergence', respectively.
Our theorem establishes these equivalences in full generality.

The Morse and contraction properties make sense for \emph{subspaces} of $X$,
not just quasi-geodesics. Our main theorem is:

\begin{theorem}\label{theorem:morse_equiv_contracting} Let $Y$ be a
  subspace of a geodesic metric space $X$. Let $\almost\geqslant 0$ be such
  that $\pi_Y^\almost$ does not contain the empty set in its image.
The following are equivalent:
\begin{enumerate}
\item There exists $\mu\from [1,\infty)\times [0,\infty)\to [0,\infty)$ such that $Y$ is $\mu$--Morse.\label{item:Morse}
\item There exists $\mu'\from [1,\infty)\to [0,\infty)$ such that every continuous $(L,0)$--quasi-geodesic with endpoints on $Y$ remains in the $\mu'(L)$-neighborhood of $Y$.\label{item:continuousbiLip}
\item There exists $\rho$ such that $\pi_Y^\almost$ is $(r,\rho)$--contracting.\label{item:sublinearcontraction}
\item There exist $\rho_1$ and $\rho_2$ such that $\pi_Y^\almost$ is $(\rho_1,\rho_2)$--contracting.\label{item:contraction}
\end{enumerate}
Moreover, in each implication we bound the parameters of the
conclusion in terms of the parameters of the hypothesis, independent
of $Y$.
\end{theorem}

Divergence, on the other hand, is specialized to quasi-geodesics.

\begin{theorem}\label{theorem:morse_equiv_csld}
Let $\gamma$ be a quasi-geodesic in a geodesic metric space $X$.
The following are equivalent:
\begin{enumerate}
\item $\gamma$ is Morse. \label{item:morse_qgeodesic}
\item $\gamma$ has completely superlinear divergence. \label{item:csld}
\end{enumerate}
Moreover, the Morse function can be bounded in terms of the divergence
function, independent of $\gamma$.
\end{theorem}

We mention a further characterization of Morse quasi-geodesics: It can
be shown fairly easily that a quasi-geodesic $\gamma\from I\to X$ is
Morse if and only if the collection of its subsegments $\{\gamma_J\mid
J \text{ is a subinterval of } I\}$ is uniformly Morse. Moreover, the
Morse functions for $\gamma$ and for the subsegments can be bounded in
terms of one another and the quasi-geodesic constants of $\gamma$.
The quantitative nature of the equivalences in
\fullref{theorem:morse_equiv_contracting} then implies that $\gamma$
is Morse if and only if the collection of its subsegments is uniformly
contracting.

\subsection{Further applications}\label{intro:further}
We consider several important theorems about strongly contracting
projections that have appeared in the literature, and generalize them
by proving sublinear
analogues.

The first of these results is the `Bounded Geodesic Image
Property', cf \cite{MasMin00, BesFuj09}. 
This says that if
$\pi_Y^\almost$ is strongly contracting then there exist
constants $A$ and $B$ such that if $\gamma$ is a geodesic segment with
$d(\gamma,Y)\geqslant A$, then $\diam \pi^\almost_Y(\gamma)\leqslant B$.
In fact, this property is equivalent to strong contraction.
We prove, in \fullref{thm:git}, that $\pi_Y^\almost$ is
$(r,\rho)$--contracting if and only if there exist a constant $A$ and
a function $\rho'\asymp\rho$ such that if $\gamma$ is a geodesic segment
with $d(\gamma,Y)\geqslant A$ then 
\[\diam
\pi_Y^\almost(\gamma)\leqslant\rho'(\max\{d(x,Y),d(x',Y)\}),\] where
$x$ and $x'$ are the endpoints of $\gamma$.

The second strong contraction result is one of the `Projection Axioms' of
Bestvina, Bromberg, and Fujiwara \cite{BesBroFuj15}.
It says that if $\pi_Y^\almost$ and
$\pi_{Y'}^{\almost'}$ are both strongly contracting, and if $Y$ and
$Y'$ are sufficiently far apart, then $\diam \pi_Y^\almost(Y')$ and
$\diam \pi_{Y'}^{\almost'}(Y)$ are bounded in terms of the contraction
constants. 
In \fullref{prop:sublinearprojection} we prove that if `strongly
contracting' is weakened to `$(r,\rho)$--contracting' then $\diam \pi_Y^\almost(Y')$ and
$\diam \pi_{Y'}^{\almost'}(Y)$ are bounded by an affine function of $\rho(d(Y,Y'))$. 
This is the best that can be expected, since even for a single point
$x$ we can only conclude $\diam\pi_Y^\almost(x)\leqslant \rho(d(x,Y))$.

Finally, a theorem of Masur and Minsky \cite{MasMin99} says,
approximately and in our language, that if for
every pair of points in a geodesic metric space $X$ there exists a path between them such that these
paths all admit semi-strongly contracting projections, with
contraction constants uniform over the family of paths, then the space
$X$ is hyperbolic.
Our \fullref{corollary:uniformcontractionimplieshyperbolic} says the
conclusion still holds if `semi-strongly contracting' is weakened to
`sublinearly contracting'.

\subsection{Robustness}
In \fullref{sec:robust} we investigate the following question:
Let $\pi_Y^\almost$ be $(\rho_1,\rho_2)$--contracting. 
What effect does changing $\rho_1$, $\almost$, or $Y$ have on this
property, in terms of $\rho_2$?

We obtain optimal answers when $\rho_1(r)=r$, see \fullref{lem:almost} and \fullref{lem:strcontbddHdist}.
It would be interesting to have good quantitative results in more
general cases.

The Morse property is invariant under quasi-isometry, so, by
\fullref{theorem:morse_equiv_csld}, the property of being sublinearly
contracting is also a quasi-isometry invariant. 
Very little is known, however, about how the contraction parameters
vary under quasi-isometry. 
In a subsequent paper \cite{ArzCasGrua} we demonstrate that strong
contraction is not preserved by quasi-isometries.

\section{Preliminaries}\label{sec:prelim}
Let $N_r(y):=\{x\in X\mid d(x,y)<r\}$ and $\overline{N}_r(y):=\{x\in
X\mid d(x,y)\leqslant r\}$. If $Y$ is a subspace of $X$, let
$N_r(Y):=\cup_{y\in Y} N_r(y)$, and $\overline{N}_r(Y):=\cup_{y\in Y}\overline{N}_r(y)$.

Let $\diam Y:=\sup \{d(y,y')\mid y,\,y'\in Y\}$.

A \emph{geodesic} is an isometric embedding of an interval.
A metric space $X$ is \emph{geodesic} if for every pair of points
$x,x'\in X$ there exists a geodesic connecting them.

The \emph{Hausdorff distance} between non-empty subspaces $Y$ and $Z$ of
$X$ is the infimal $C$ such that $Y\subset \overline{N}_C(Z)$ and $Z\subset\overline{N}_C(Y)$.
Two subspaces are \emph{$C$--Hausdorff equivalent} if
the Hausdorff distance between them is at most $C$.

Given $L\geqslant 1$ and $A\geqslant 0$, a map $\phi\from X\to Y$ between metric spaces is an
\emph{$(L,A)$--quasi-isometric embedding} if
$\frac{1}{L}d(x,x')-A\leqslant
d(\phi(x),\phi(x'))\leqslant L d(x,x')+A$ for every $x,\,x'\in X$.
It is an \emph{$(L,A)$--quasi-isometry} if, in addition, $Y=\overline{N}_A(\phi(X))$.

An \emph{$(L,A)$--quasi-geodesic} is an $(L,A)$--quasi-isometric
embedding of an interval.

\begin{definition}\label{def:sublinear}
A function $f$ is \emph{sublinear} if it is non-decreasing,
eventually non-negative, and $\lim_{r\to\infty}\frac{f(r)}{r}=0$.
\end{definition}

We write $f\asympleq g$ if there exist  constants $C_1>0$,
$C_2>0$, $C_3\geqslant 0$, and $C_4\geqslant 0$  such
that $f(r)\leqslant C_1g(C_2r+C_3)+C_4$ for all $r$.
This partial order gives an equivalence relation $f\asymp g$ if
$f\asympleq g$ and $f\asympgeq g$. 
If $f\asymp g$ we say $f$ and $g$ are \emph{asymptotic}.

\section{Examples of contraction}\label{sect:examples}
We begin with a classical example.

\begin{example}
\begin{figure}[h]
  \centering
\labellist
\small
\pinlabel $Y$ [tl] at 106 49
\pinlabel $H$ at 50 130
\endlabellist
  \includegraphics[scale=.7]{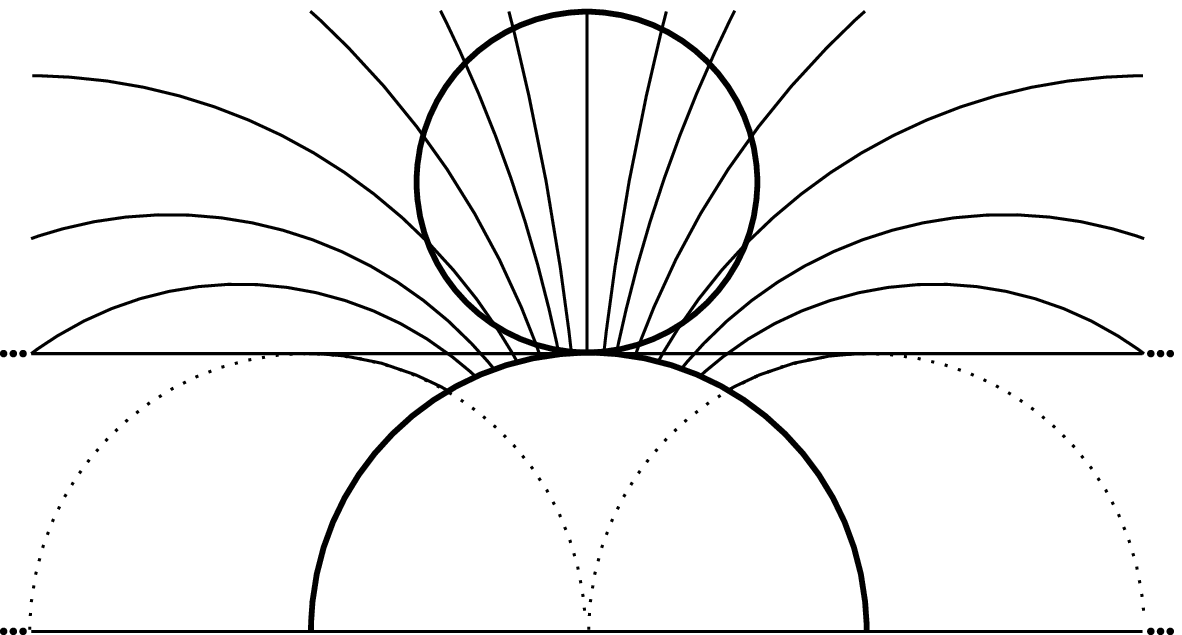}
  \caption{Contraction in $\mathbb{H}^2$.}
  \label{fig:h2}
\end{figure}
Let $X$ be the hyperbolic plane, with the upper half-space model, and let $Y$ be the geodesic that is the upper half of the unit circle, see \fullref{fig:h2}.
Pick any point $x\notin Y$. 
Up to isometry, we may assume $x$ sits on the $y$--axis above $Y$.
The ball of radius $d(x,Y)$ about $x$ is contained in the horoball $H:=\{(a,b)\in\mathbb{R}^2\mid b\geqslant 1\}$.
The closest point projection of $H$ to $Y$ has diameter $\ln (3+2\sqrt{2})$. 
Thus, $\pi_Y^0$ is $(r,\ln (3+2\sqrt{2}))$--contracting.
\end{example}

We now construct examples exhibiting a wider range of contracting
behaviors than have appeared previously in the literature.

\begin{example}\label{ex:log}
Let $\rho=\rho_1\from [0,\infty)\to [0,\infty)$ be an unbounded function such that
$\rho(r)\leqslant r$, $\mathrm{Id}-\rho$ is unbounded, and there exists an $A\geqslant 0$ with $\rho(A)>0$ such
that $0\leqslant\rho(a+b)-\rho(a)<b$ for all $a\geqslant A$ and $b>0$.  
We construct a space $X$ and $Y\subset X$ such that $\pi_Y^0$ is
$(\rho,2)$--contracting but not strongly contracting.

The map $\phi\from [A,\infty) \to [A-\rho(A),\infty) : x\mapsto
x-\rho(x)$ is a bijection by assumption. 
We set $\sigma(0):=\phi(A)$ and, for $i\in\N$, recursively define\footnote{
An \emph{Abel function} for $f$ is a function $\alpha$ such that $\alpha(f(x))=\alpha(x)+1$. 
The function $\sigma$ is the inverse of an Abel function for $\phi^{-1}$.}
$\sigma(i+1):=\phi^{-1}(\sigma(i))$. 
This is well-defined since $[A,\infty)\subset [A-\rho(A),\infty)$. 
Rearranging this expression yields
$\rho(\sigma(i+1))=\sigma(i+1)-\sigma(i)$. 
Note that $\sigma(i+1)-\sigma(i)\geqslant\rho(A)>0$ for every $i\in\N\cup\{0\}$, whence, in particular, $\sigma(i)\to\infty$ as $i\to\infty$.

Let $Y:=[0,\infty)$ be a ray.
For $i\in\N\cup\{0\}$, let $Z_i$ be a segment of length $\sigma(i)$ with endpoints labelled $y_i$ and $z_i$.
Identify $y_i$ with the point $i$ in $Y$.
Let $W_i$ be a segment of length $\sigma(i+1)-\sigma(i)+1$
connecting $z_i$ to $z_{i+1}$. 
Let $X$ be the resulting geodesic metric space.
See \fullref{fig:log}.

\begin{figure}[h]
  \centering
\labellist
\small
\pinlabel $Y$ [r] at 0 25
\pinlabel $Z_0$ [b] at 18 0
\pinlabel $Z_1$ [b] at 33 17
\pinlabel $Z_2$ [b] at 63 34
\pinlabel $Z_3$ [b] at 122 51
\pinlabel $W_0$ [tl] at 45 6
\pinlabel $W_1$ [tl] at 83 22
\pinlabel $W_2$ [tl] at 188 37
\tiny
\pinlabel $y_0$ [tr] at 2 2
\pinlabel $z_0$ [tl] at 32 0 
\pinlabel $x_0$ [t] at 58 12
\pinlabel $z_1$ [tl] at 65 15
\pinlabel $x_1$ [t] at 121 29
\pinlabel $z_2$ [tl] at 130 30
\pinlabel $x_2$ [t] at 246 44
\pinlabel $z_3$ [tl] at 258 45
\pinlabel $y_3$ [r] at 2 47 
\endlabellist
  \includegraphics{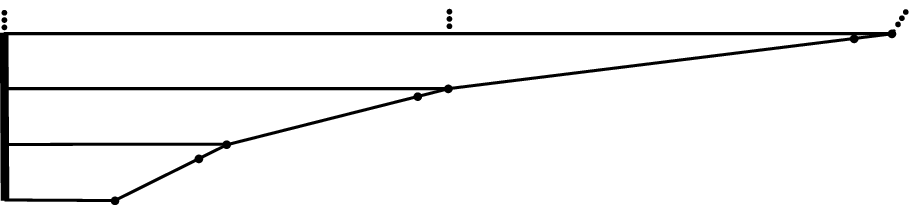}
  \caption{$(\rho_1,2)$--contraction}
  \label{fig:log}
\end{figure}

Let $x_i$ be the point of $W_i$ at distance $1/2$ from $z_{i+1}$. Clearly $\diam \pi_Y^0(x_i)=1$. 
It is easy to see that each complementary component of $X\setminus
(Y\cup\{x_i\}_{i\in\N\cup\{0\}})$ projects to a single point of $Y$. 
Now consider the ball of radius $\rho(d(x,Y))$ about some $x$. First
assume $x\in W_i$ for some $i$. 
Our assumptions on $\rho$ yield:
 \[\overline{N}_{\rho(d(x,Y))}(x)\subseteq W_i\cup
 \overline{N}_{\rho(\sigma(i))+1/2}(z_i)\cup
 N_{\rho(\sigma(i+1))+1/2}(z_{i+1})\]
The latter may contain $x_i$ and $x_{i+1}$ but no other $x_j$. If, on the other hand, $x$ is in some $Z_i$, then $\overline{N}_{\rho(d(x,Y))}(x)$ is contained in $Z_i\cup\overline{N}_{\rho(d(z_i,Y))}(z_i)$. Therefore, for any $x\in X$, we have that $\pi_Y^0(\overline{N}_{\rho(d(x,Y))}(x))$ has diameter at most 2.

Observe that $\overline{N}_{d(z_i,Y)}(z_i)$ contains $\{z_j,z_{j+1},\dots z_i\}$ for $0\leqslant i-j\leqslant\sigma(j)$. Since $\sigma(i)\to\infty$ as $i\to\infty$, this implies that $Y$ is not strongly contracting.

Concrete examples include:
\begin{itemize}
\item $\rho(r):=2\sqrt{r}-1$ and $A=1$ and $\sigma(r):=r^2$.
\item $\rho(r):=r/2$ and $A=2$ and $\sigma(r):=2^r$. This is an
  example of
  semi-strong contraction. 
\item $\rho(r):=\min\{r,r-\log_2 r\}$ and $A=2$ and
  $\sigma(r):=2\uparrow\uparrow r$.
\end{itemize}
In Knuth's `up-arrow notation' $2\uparrow\uparrow r$ denotes tetration, so that $2\uparrow\uparrow r=\underbrace{2^{\cdot^{\cdot^{\cdot^2}}}}_{r\text{ times}}$ when $r\in\N\cup\{0\}$.
\end{example}

The following proposition shows that it is sometimes possible to
`trade' between the input and output contraction functions, so we can
use \fullref{ex:log} to demonstrate further examples of 
$(\rho_1,\rho_2)$--contraction conditions.
\begin{proposition}\label{prop:Abel}
  Suppose that $\pi_Y^\almost$
  is $(\rho_1,B)$--contracting, where $B$ is a constant and
  $\rho=\rho_1$ is a non-decreasing, non-negative, unbounded function
  such that $\mathrm{Id}-\rho$ is unbounded and
such that there exists a constant $A$ such that $\rho(A)>0$ and $0\leqslant
\rho(a+b)-\rho(a)<b$ for all $a\geqslant A$ and $b>0$.
Define $A':=A-\rho(A)$.
For $x\in [A',\infty)$ define\footnote{The function $\alpha\from [A',\infty)\to \N\cup\{0\}$
is an Abel function for $(\mathrm{Id}-\rho)^{-1}$.
For instance,  take $\alpha$ to be the inverse of $\sigma\from\N\cup\{0\}\to\sigma(\N\cup\{0\})$ from \fullref{ex:log} extended to all of $[A',\infty)$ by a rounding-off function.}
$\alpha(x)$ to be the minimal non-negative
 integer such that $(\mathrm{Id}-\rho)^{\alpha(x)}(x)\in
[A',A)$.
Then $\pi_Y^\almost$ is $(r-A,\rho_2)$--contracting for some $\rho_2\asymp \alpha$.
\end{proposition}

\begin{proof}
Observe as in \fullref{ex:log} that the map $\phi\from x\mapsto
x-\rho(x)$ is a bijection $[A,\infty)\to[A',\infty)$ and that, since
$\phi$ is strictly increasing for $x\geqslant A$, the collection
$\{[\phi^{k}(A'),\phi^{k-1}(A')) \mid k\leqslant 0\}$ is a partition of $[A',\infty)$.

We show that $\rho_2(r):=B\alpha(r)$ will suffice.
It follows from unboundedness of $\rho$ that $\rho_2$ is sublinear: we
have $\rho_2\asymp\alpha$. 
The map $\alpha$ is a step function with steps of height 1, so it is sufficient to show that the lengths of the steps go to infinity, ie $\phi^{-n-1}(A')-\phi^{-n}(A')\to \infty$ as $n\to\infty$.
As computed in \fullref{ex:log}, we have $\rho(\phi^{-n-1}(A'))=\phi^{-n-1}(A')-\phi^{-n}(A')$. Since $\phi^{-n-1}(A')\to\infty$ as $n\to\infty$ as argued in \fullref{ex:log} and since $\rho$ goes to infinity, sublinearity follows.

Let $x$ and $y$ be points of $X$ such that $d(x,y)\leqslant d(x,Y)-A$. 
Define $r_0:=d(x,Y)$ and while $r_0-r_i\leqslant d(x,y)$, define
$r_{i+1}:=\phi(r_i)$. 
Note that this is well-defined, ie $r_i\geqslant A$, since
$r_0-r_i\leqslant d(x,y)\leqslant r_0-A$. 
Let $k$ be the largest index such that $r_0-r_k\leqslant d(x,y)$. Then the fact that $\phi^{\alpha(r_0)}(r_0)<A$ and the observation we just made shows $k<\alpha(r_0)$.

Fix a geodesic from $x$ to $y$ and for $0\leqslant i\leqslant k$ define $x_i$ to be the point at distance $r_0-r_i$ from $x$ along this geodesic. Define $x_{k+1}:=y$. For $0\leqslant i\leqslant k$ we have $d(x_{i+1},x_i)\leqslant \rho(d(x_i,Y))$ by construction, whence:  
 \[\diam\pi_Y^\almost(x)\cup\pi_Y^\almost(y)\leqslant \sum_{i=0}^{
  k}\diam\pi_Y^\almost(x_i)\cup\pi_Y^\almost(x_{i+1})\leqslant B
  \alpha(r_0)\]
Thus, $\pi_Y^\almost$ is $(r-A,\rho_2)$--contracting.
\end{proof}

Applying \fullref{prop:Abel} to the concrete examples in
\fullref{ex:log} we see:
\begin{itemize}
\item $(2\sqrt{r}-1,2)$--contracting implies $(r-1,\rho_2)$--contracting
  for $\rho_2\asymp \sqrt{\cdot}$.
\item $(r/2,2)$--contracting implies $(r-2,\rho_2)$--contracting for
$\rho_2\asymp \log_2$.
\item Finally, $(r-\log_2r,2)$--contracting implies 
$(r-2,\rho_2)$--contracting for $\rho_2\asymp \mathrm{superlog}_2$.
\end{itemize}

That the converse to \fullref{prop:Abel} can fail follows from the next example.

\begin{example}\label{ex:necklace}
Let $\rho_2$ be a sublinear function such that $0<\rho_2(r)<r$.
Let $Y$ be a line.
Choose a collection of disjoint intervals $\{I_i\}_{i\in\N}$ of $Y$
such that $|I_i|=\rho_2(i)$ and let $y_i$ be the center of $I_i$.
Connect the endpoints of $I_i$ by attaching a segment $J_i$ of length $4i$, and let
$x_i$ be the center of this segment. Let $X$ be the resulting geodesic
space, see \fullref{fig:necklace}.
We claim $\pi_Y^0$ is $(r,\rho_2)$--contracting.

\begin{figure}[h]
  \centering
  \labellist
\small
\pinlabel $Y$ [tr] at 75 0
\tiny
\pinlabel $x_1$ [b] at 29 32
\pinlabel $x_2$ [b] at 121 44
\pinlabel $x_3$ [b] at 241 63
\pinlabel $y_1$ [t] at 29 1
\pinlabel $y_2$ [t] at 121 1
\pinlabel $y_3$ [t] at 241 1
\endlabellist
\includegraphics[scale=.75]{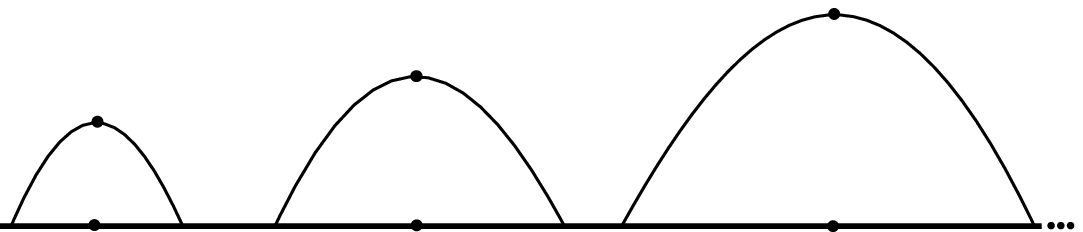}
  \caption{$(r,\rho_2)$--contracting}
  \label{fig:necklace}
\end{figure}

Suppose that $x\in J_i\subset X$ and $d(x,Y)<i$. 
Then $d(x,x_i)>d(x,Y)$, and $\diam \pi_Y^0(\overline{N}_{d(x,Y)}(x))=0$.
For $x\in J_i\subset X$ with $d(x,Y)\geqslant i$ we have $d(x,x_i)\leqslant d(x,Y)$ and:
\[\diam \pi_Y^0(\overline{N}_{d(x,Y)}(x))=\diam
\pi_Y^0(x_i)=\rho_2(i)\leqslant \rho_2(d(x,Y))\]
This proves the claim.
Furthermore,  $\rho_2$ is optimal, in the following sense: Since $\diam \pi_Y^0(x_i)=\rho_2(d(x_i,Y)/2)=\rho_2(i)$, if $\rho'_1$ and
$\rho'_2$ are some other functions such that $\pi_Y^0$ is
$(\rho_1',\rho_2')$--contracting then $\rho_2(i)\leqslant\rho_2'(2i)$ for $i\in\N$.
\end{example}

\section{The Morse property}\label{sec:morse}
The following two propositions establish our main result, \fullref{theorem:morse_equiv_contracting}.

\begin{proposition}\label{prop:basiccontractionimpliesmorse}
Let $Y$ be a subspace of a geodesic metric space $X$.
Suppose $\pi_Y^\almost$ is $(\rho_1,\rho_2)$--contracting.
There exists a function $\mu$, depending only on $\almost$, $\rho_1$, and
$\rho_2$, such that $Y$ is $\mu$--Morse.
\end{proposition}
\begin{proof}
Given $L'$ and $A'$ there exist $L$, $A$, and $C$ such that every
$(L',A')$--quasi-geodesic is $C$--Hausdorff equivalent to a
continuous $(L,A)$--quasi-geodesic with the same endpoints \cite[Lemma
III.H.1.11]{BriHae99}.
Thus, it suffices to show there exists a bound $B$, depending only on
$\almost$, $\rho_1$ and $\rho_2$, such that every continuous
  $(L,A)$--quasi-geodesic connecting points on $Y$ is contained in
  $N_B(Y)$.
Then we set $\mu(L',A'):=B+C$.

Let $\gamma$ be a continuous $(L,A)$--quasi-geodesic with endpoints on $Y$.
Take $E$ to be sufficiently large so that
$\rho_1(E)>3A$ and for all $r\geqslant E$ we have
$\frac{\rho_2(r)}{\rho_1(r)}<\frac{1}{3L^2}$.

Suppose $\gamma\not\subseteq N_E(Y)$, and let $[a,b]$ be a maximal subinterval of the domain of $\gamma$ such that $\gamma|_{[a,b]}\subset X\setminus N_E(Y)$. 
We show there exists a $T$ independent of $\gamma$ and $Y$ such that
$b-a\leqslant T$.
We conclude by setting  $B:=E+L\cdot\frac{T}{2}+A$.

Let $t_0:=a$. 
Supposing we have defined $t_0,\dots,t_i$, if $d(\gamma(t_i),\gamma(b))>\rho_1(d(\gamma(t_i),Y))$ define $t_{i+1}$ to be the first time that $d(\gamma(t_i),\gamma(t_{i+1}))=\rho_1(d(\gamma(t_i),Y))$.
Such a $t_{i+1}$ exists because $\gamma$ is continuous. 
Since $d(\gamma,Y)\geqslant E$ we have
$d(\gamma(t_i),\gamma(t_{i+1}))=\rho_1(d(\gamma(t_i),Y))\geqslant \rho_1(E)>0$, so
after finitely many steps we reach an index $k$ such that $d(\gamma(t_k),\gamma(b))\leqslant\rho_1(d(\gamma(t_k),Y))$.
Applying the contraction condition to the points $\gamma(t_i)$, we see: 
\[
\diam
\pi_Y^\almost(\gamma(a))\cup\pi_Y^\almost(\gamma(b))\leqslant
\sum_{i=0}^k\rho_2(d(\gamma(t_i),Y))\]

This allows us to estimate:
\begin{align}
  d(\gamma(a),\gamma(b)) &\leqslant
                           d(\gamma(a),\pi_Y^\almost(\gamma(a)))\notag\\
&\qquad+\diam  \pi_Y^\almost(\gamma(a))\cup\pi_Y^\almost(\gamma(b)) +d(\gamma(b),\pi_Y^\almost(\gamma(b)))\notag\\
&\leqslant 2(E+\almost)+\sum_{i=0}^{k}\rho_2(d(\gamma(t_i),Y))\label{diamest}
\end{align}

On the other hand, since $\gamma$ is a $(L,A)$--quasi-geodesic, we have:
\begin{align*}
  L d(\gamma(a),\gamma(b))+LA&\geqslant b-a= b-t_k+\sum_{i=0}^{k-1}(t_{i+1}-t_i)\\
&\geqslant
  \frac{1}{L}(d(\gamma(b),\gamma(t_k))-A)+\sum_{i=0}^{k-1}\frac{1}{L}(d(\gamma(t_{i+1}),\gamma(t_i))-A)\\
&= \frac{1}{L}(d(\gamma(b),\gamma(t_k))-\rho_1(d(\gamma(t_k),Y)))\\
&\qquad+\sum_{i=0}^{k}\frac{1}{L}(\rho_1(d(\gamma(t_i),Y))-A)\\
&\geqslant
  \frac{-d(\gamma(b),Y)}{L}+\sum_{i=0}^{k}\frac{1}{L}(\rho_1(d(\gamma(t_i),Y))-A)\\
&= -\frac{E}{L}+\sum_{i=0}^{k}\frac{1}{L}(\rho_1(d(\gamma(t_i),Y))-A)
\end{align*}

Combining this with the previous inequality and rearranging terms, we have:
\[ \sum_{i=0}^k\left(\rho_1(d(\gamma(t_i),Y))-L^2\rho_2(d(\gamma(t_i),Y))-A\right)\leqslant E+L^2A+2L^2(E+\almost)\]

Now,  left hand side is at least
$L^2\sum_{i=0}^k\rho_2(d(\gamma(t_i),Y))$, by our choice of $E$;
combined with \eqref{diamest}, this gives us:
\[d(\gamma(a),\gamma(b))\leqslant \frac{E}{L^2}+A+4(E+\almost)\]

This estimate and the fact that $\gamma$ is a quasi-geodesic
give us a bound for $b-a$.
\end{proof}

\begin{proposition}\label{prop:morseimpliessublinearcontraction}
Let $Y$ be a subspace of a geodesic metric space $X$.
Suppose there is a non-decreasing function $\mu$ such that every continuous
$(L,0)$-quasi-geodesic with endpoints on $Y$ is contained in the
closed $\mu(L)$--neighborhood of $Y$. 
Suppose the empty set is not in the image of $\pi_Y^\almost$.
Then there is a function $\rho'$, depending only on $\mu$ and $\almost$, such that $\pi_Y^\almost$
is $(r,\rho')$--contracting. 
\end{proposition}
We remark that since an $(L,0)$--quasi-geodesic is also an
$(L',0)$--quasi-geodesic for any $L'>L$, there is no loss in requiring
the Morse function to be non-decreasing. 
\begin{proof}
Consider the optimal contraction function:
\[\rho(r):=\sup_{d(x,y)\leqslant d(x,Y)\leqslant
  r}\diam\pi_Y^\almost(x)\cup\pi_Y^\almost(y)\leqslant 4r+2\almost\]
Our goal is define a function $\rho'$ depending on
$\mu$ and $\almost$ that is
non-negative, non-decreasing, and sublinear and such that $\rho'$
is an upper bound for $\rho$.

Define $\rho'(r):=0$ if $\almost =0$ and $\mu\equiv 0$. In this case
$\rho'$ clearly has the first three properties.
Otherwise, we first replace $\mu$ by $s\mapsto
\inf_{t>s}\mu(s)$. 
The new $\mu$ still satisfies the hypotheses of the proposition and
has that additional property that it is right continuous: 
$\lim_{t\to s^+}\mu(t)=\mu(s)$ for all $s\geqslant 1$.
Define $\rho'(0):=2\almost$ and for $r>0$ define:
\[\rho'(r):=\sup\left\{s\leqslant 4r+2\almost\mid s\leqslant 18\mu\left(\frac{3(4r+2\almost)}{s}\right)+12\almost\right\}\]

If $\mu\equiv 0$ then $\rho'$ increases linearly from $2\almost$ to
$12\almost$ and then remains constant, so it is non-negative, non-decreasing, and
sublinear.

If $\mu\not\equiv 0$ then $\rho'(r)>0$ when $r>0$, and the conditions
on $\mu$ ensure $\rho'$ is actually a maximum.
The fact that it is non-decreasing then follows by observing that
$\rho'(r)$ participates in the supremum defining $\rho'(r')$ when
$0\leqslant r<r'$.
To see $\rho'$ is sublinear, we suppose that
$\limsup_{r\to\infty}\rho'(r)/r>0$ and derive a contradiction. 
Suppose that there exists some $\delta>0$ and
a sequence $(r_i)$ of positive numbers increasing without bound such that
$\rho'(r_i)>\delta r_i$ for all $i$.
By definition of $\rho'$, for each $i$ there exists $\delta r_i<s_i\leqslant 4r_i+2\almost$
such that: 
\[s_i\leqslant
18\mu\left(\frac{3(4r_i+2\almost)}{s_i}\right)+12\almost\leqslant
18\mu\left(\frac{3(4r_i+2\almost)}{\delta r_i}\right)+12\almost\]
This is a contradiction, since the left-hand side grows without bound
while the right-hand side is bounded above by
$18\mu(\frac{12}{\delta}+1)+12\almost$ once $i$ is sufficiently large.

\medskip
Now we must show $\rho(r)\leqslant \rho'(r)$. 
It suffices to check this for those $r$ such that $\rho(r)>0$.
The idea of the proof is to choose, for each such $r$,
points $x$ and $y$ such that $d(x,y)\leqslant d(x,Y)\leqslant r$ whose
projection diameters nearly realize $\rho(r)$.
Take a path $\gamma$ that is a concatenation of geodesics from a
projection point of $x$ to $x$, then from $x$ to $y$, then from $y$ to
a projection point of $y$. 
For $L:=\frac{3(4r+2\almost)}{\rho(r)}\geqslant 3$ we show that we can make $\gamma$ into an
$(L,0)$--quasi-geodesic $\gamma'$ by introducing at most two
shortcuts in a particular way.
The Morse hypothesis implies that $\gamma'$ is contained in the $\mu(L)$--neighborhood
of $Y$.
We then argue that the condition $d(x,y)\leqslant d(x,Y)$
implies:
\begin{equation}
  \label{eq:1}
\rho(r)< 18\mu(L)+12\almost
\end{equation}

In the case that $\almost=0$ and $\mu\equiv 0$, this gives a
contradiction, which means that there is no $r$ for which $\rho$ takes
a positive value, and we have $\rho(r)=\rho'(r)= 0$ for all $r$.
Otherwise, plugging  the value of $L$ into \eqref{eq:1}, we conclude that
$\rho(r)$ participates in the supremum defining $\rho'(r)$, whence
$\rho(r)\leqslant \rho'(r)$.

\bigskip
First we show how to produce quasi-geodesics.
Consider points $x$, $y$, $p_x\in\pi_Y^\almost(x)$, and $p_y\in\pi_Y^\almost(y)$.
Let $\gamma:=[p_x,x][x,y][y,p_y]$ be a concatenation of three
geodesics.
Let $[p,q]_\gamma$ denote the subsegment of $\gamma$ from $p$ to $q$,
and let $|[p,q]_\gamma|$ denote its length.
For this part of the argument we may use any $L\geqslant\frac{|\gamma|}{d(p_x,p_y)}\geqslant 1$.
Consider the continuous function $D(p,q):=Ld(p,q)-|[p,q]_\gamma|$ 
defined on points $(p,q)\in \gamma\times\gamma$ such that $p$
precedes $q$ on $\gamma$.
The restriction on $L$ implies that $D(p_x,p_y)\geqslant 0$.
We conclude that if $[p,q]_\gamma$ is a subsegment of
$\gamma$ that is maximal with respect inclusion among subsegments for
which $D$ takes
non-positive values on the endpoints, then $Ld(p,q)=|[p,q]_\gamma|$.
We consider several cases. 
Each carries the additional assumption that
we are not in one of the previous cases.

\medskip
 {\itshape Case 0: $D$ is non-negative.} Set $\gamma':=\gamma$, which is an
 $(L,0)$--quasi-geodesics by definition of $D$.

\medskip
{\itshape Case 1: $D$ takes a non-positive value on $[p_x,x]_\gamma\times
  [y,p_y]_\gamma$.} In this case there exist points $x'\in [p_x,x]$
and $y'\in [p_y,y]$ such that the segment $[x',y']_\gamma$ is maximal
with respect to inclusion among subsegments of $\gamma$ with  the property that $D$ takes non-positive values on
endpoints. 
Define $\gamma'$ by replacing
$[x',y']_\gamma$ by some geodesic segment with the same endpoints;
$\gamma':=[p_x,x']_\gamma[x',y'][y',p_y]_\gamma$.
We claim that $\gamma'$ is an $(L,0)$--quasi-geodesic.
Since $\gamma'$ is a concatenation of geodesic segments, it suffices to check
that points on distinct segments are sufficiently far apart. 
We check distances between arbitrary points $x''\in [p_x,x']_{\gamma'}$, $z\in [x',y']_{\gamma'}$, and
$y''\in [y',p_y]_{\gamma'}$.

Suppose, for contradiction, that $Ld(x'',y'')<|[x'',y'']_{\gamma'}|$.
Since $[x',y']_\gamma$ has been replaced by a geodesic segment,
$Ld(x'',y'')<|[x'',y'']_{\gamma'}|\leqslant |[x'',y'']_\gamma|$, so $D(x'',y'')<0$.
Since $D(x',y')=0$ we have $x''\in [p_x,x')_\gamma$ or $y''\in
(y',p_y]_\gamma$, but then $[x'',y'']_\gamma$ is a subsegment of
$\gamma$ strictly containing $[x',y']_\gamma$ such that $D$ takes a
non-positive value on its endpoints. 
This contradicts maximality of
$[x',y']_\gamma$ among such subsegments, so $d(x'',y'')\geq\frac{|[x'',y'']_{\gamma'}|}{L}$.

Suppose, for contradiction, that $Ld(x'',z)<|[x'',z]_{\gamma'}|$. 
This implies $x''\neq x'$, because $x'$ and $z$ lie on a geodesic
subsegment of $\gamma'$.
We estimate:
\begin{align*}
  d(x'',y')&\leqslant d(x'',z)+d(z,y')\\
&<\frac{|[x'',z]_{\gamma'}|}{L}+d(z,y')\\
&=\frac{d(x'',x')+d(x',z)}{L}+d(x',y')-d(x',z)\\
&=\frac{|[x'',x']_\gamma|}{L}+\frac{|[x',y']_\gamma|}{L}-\left(\frac{L-1}{L}\right)d(x',z)\\
&\leqslant \frac{|[x'',y']_\gamma|}{L}
\end{align*}
Since $x''\in [p_x,x')_\gamma$, we have exhibited a
subsegment $[x'',y']_\gamma$ strictly containing $[x',y']_\gamma$
such that $D$ takes a non-positive value on its endpoints.
This contradicts maximality of $[x',y']_\gamma$ among such subsegments, so $d(x'',z)\geqslant
\frac{|[x'',z]_{\gamma'}|}{L}$.

A symmetric argument shows
$d(y'',z)\geqslant \frac{|[y'',z]_{\gamma'}|}{L}$, so $\gamma'$ is an $(L,0)$--quasi-geodesic.

\medskip
{\itshape Case 2: $D$ takes a non-positive value on an
  element of 
  $[p_x,x]_\gamma\times (x,y]_\gamma$.}
Let $[x',q_x]_\gamma$ be a subsegment of $\gamma$ maximal with respect
to inclusion among subsegments for which $D$
takes non-positive values on endpoints, with $x'\in [p_x,x]_\gamma$. 
Since we are not in Case 1, $q_x\in (x,y)_\gamma$.
Now consider whether or not $[q_x,p_y]_\gamma$
is an $(L,0)$--quasi-geodesic.
If so, define $\gamma':=[p_x,x']_\gamma[x',q_x][q_x,p_y]_\gamma$.
Otherwise, $D$ takes a negative value on an element
of $[q_x,y)_\gamma\times (y,p_y]_\gamma$. 
Let $[q_y,y']_\gamma$ be a maximal subsegment of $[q_x,p_y]_\gamma$,
with $q_y\in [q_x,y)_\gamma$ and $y'\in (y,p_y]$ on which
$D$ takes non-positive values on endpoints.
We claim that  $q_y\in (q_x,y)_\gamma$ and $D(q_y,y')=0$, because if
$D(q_y,y)<0$ and $q_y\neq q_x$ then we can enlarge the subsegment,
contradicting maximality, while if $q_y=q_x$ then $D(x',y')\leqslant 0$,
contradicting the assumption that we are not in Case 1.

In either of these cases, we claim $\gamma'$ is an
$(L,0)$--quasi-geodesic. 
This follows by verifying that the distance between
points in distinct geodesic components of $\gamma'$ have distance at
least equal to the length of the subsegment of $\gamma'$ they bound
divided by $L$. 
The strategy is to suppose $D$ attains a strictly
negative value and then either derive a contradiction to maximality of
$[x',q_x]_\gamma$ or $[q_y,y']_\gamma$ or to the assumption that we
are not Case 1.
The arguments are substantially similar to the computations in Case 1
and are left to the reader.

\medskip
{\itshape Case 3: $D$ takes a non-positive value on an
  element of 
  $[x,y)_\gamma\times [y,p_y]_\gamma$.}
The argument here is symmetric to the subcase of Case 2 in which only
a corner at $x$ is cut short.

\bigskip
We have shown how to
produce an $(L,0)$--quasi-geodesic $\gamma'$ from $\gamma$.
We now proceed to show $\rho(r)\leqslant \rho'(r)$ for any $r$ such
that $\rho(r)>0$.
Since $\rho(r)>0$ there exist $x$ and $y$ such that $d(x,y)\leqslant
d(x,Y)\leqslant r$ and $\diam\pi_Y^\almost(x)\cup\pi_Y^\almost(y)>\frac{2}{3}\rho(r)$.
Choose $p_x\in\pi_Y^\almost(x)$, $p_y\in\pi_Y^\almost(y)$ such that
$d(p_x,p_y)>\frac{2}{3}\rho(r)$.

Let $\gamma:=[p_x,x][x,y][y,p_y]$.
Let $L:=\frac{12r+6\almost}{\rho(r)}\geqslant
2\frac{|\gamma|}{d(p_x,p_y)}$, and let $\gamma'$ be the
$(L,0)$--quasi-geodesic produced from $\gamma$ as above. 
By the Morse hypothesis, $\gamma'$ is contained in the $\mu(L)$--neighborhood of
$Y$. 

\medskip
{\itshape Case a: $\gamma'$ comes from Case 0 or Case 3.} 
In this case $x\in\gamma'$, so  $d(x,Y)\leqslant \mu(L)$, so
$\rho(r)<\frac{3}{2}d(p_x,p_y)\leqslant \frac{3}{2}(4\mu(L)+2\almost)$.

\medskip
{\itshape Case b: $\gamma'$ comes from Case 1.} 
In this case $p_x\in\pi_Y^\almost(x')$ and $p_y\in\pi_Y^\almost(y')$, so
$d(x',p_x)\leqslant\mu(L)+\almost$ and $d(y',p_y)\leqslant \mu(L)+\almost$.
Also, by definition of $L$ we have:
\[d(x',y')=\frac{|[x',y']_\gamma|}{L}\leqslant\frac{|\gamma|}{L}\leqslant\frac{4r+2\almost}{\frac{3(4r+2\almost)}{\rho(r)}}=\frac{\rho(r)}{3}\]
Since $d(p_x,p_y)>\frac{2}{3}\rho(r)$, we conclude
  $d(x',p_x)+d(y',p_y)>\frac{\rho(r)}{3}$, so that $\rho(r)<6\mu(L)+6\almost$.

\medskip
{\itshape Case c: $\gamma'$ comes from Case 2.}
In this case $\gamma'$ contains a geodesic segment from a point $x'\in
[p_x,x]_\gamma$ to a point $q_x\in [x,y]_\gamma$.
As in the previous case,
$d(x',q_x)=\frac{|[x',q_x]_\gamma}{L}\leqslant
\frac{|\gamma|}{L}\leqslant \frac{\rho(r)}{3}$.
Consider a point $w\in\pi_Y^\almost(q_x)$.
Since $d(x,y)\leqslant d(x,Y)$, we have $d(q_x,y)\leqslant
d(q_x,Y)\leqslant \mu(L)$, which implies $d(w,p_y)\leqslant 4\mu(L)+2\almost$.
Thus $d(p_x,w)> \frac{2}{3}\rho(r)-(4\mu(L)+2\almost)$.
We also have $d(x',Y)\leqslant \mu(L)$ and $d(q_x,Y)\leqslant \mu(L)$,
since both these points belong to $\gamma'$, so:
\[\frac{\rho(r)}{3}\geqslant d(x',q_x)\geqslant
d(p_x,w)-d(x',p_x)-d(q_x,w)>\frac{2}{3}\rho(r)-(6\mu(L)+4\almost)\]
The resulting bound on $\rho(r)$ is the largest of the three cases,
and establishes the bound of \eqref{eq:1}, completing the proof.
\end{proof}

\section{Divergence}
In this section we relate divergence to contraction and the Morse
property, thereby proving \fullref{theorem:morse_equiv_csld}.

There is a link between the Morse property and superlinear divergence
via asymptotic cones \cite{DruMozSap10}. 
Although this principle is well-known, there are competing definitions of
`superlinear' and `divergence', so we give a detailed proof of
\fullref{theorem:morse_equiv_csld} in terms of our definitions.
Our analysis actually yields more.
In the introduction we claimed that for a quasi-geodesic  the Morse property, hence,
sublinear contraction, is morally the opposite of high divergence. 
We prove a precise technical  formulation of this claim in \fullref{prop:csld_implies_sublinear_contraction}.
Roughly speaking, the result we obtain is  that if divergence of a
quasi-geodesic $\gamma$ is greater than a  function
$f$ then almost closest point projection to $\gamma$ is $(r, f^{-1})$--contracting.

\begin{defi}\label{def:divergence}
Let $X$ be a geodesic metric space and let $\gamma\from\R\to X$ be an $(L,A)$--quasi-geodesic. 
Let $\lambda\in(0,1]$, and let $\kappa\geqslant L+A$.
Let $\Lambda_\gamma(r,s;L,A,\lambda,\kappa)$ be the infimal length
of a path from $\gamma(s-r)$ to $\gamma(s+r)$ that is disjoint from
the ball of radius $\lambda(L^{-1}r - A) -\kappa$ centered at
$\gamma(s)$, or $\infty$ if no such path exists.
The \emph{$(L,A,\lambda,\kappa)$--divergence} of $\gamma$ evaluated at
$r$ is $\Delta_\gamma(r;L,A,\lambda,\kappa):=\inf_s \Lambda_\gamma(r,s;L,A,\lambda,\kappa)$.
\end{defi}

Notice that if $\gamma$ is a geodesic, $\lambda:=1/2$, and $\kappa:=2$ we recover the definition of divergence we gave in the introduction.

We make the convention that $\infty\leqslant\infty$.

In light of the following lemma, $\gamma$ has a well defined
divergence, up to equivalence of functions, and we use
$\Delta_\gamma(r)$ to denote the equivalence class of $\Delta_\gamma(r; L,A,\lambda,\kappa)$.

\begin{lemma}\label{lemma:divergence_welldefined} 
Let $\gamma$ be an $(L,A)$--quasi-geodesic. 
Suppose $\gamma$ is also an $(L',A')$--quasi-geodesic. 
Let $\lambda,\lambda'\in(0,1]$, $\kappa\geqslant L+A$, and $\kappa'\geqslant L'+A'$. 
Then $\Delta_\gamma(r;L,A,\lambda,\kappa)\asymp\Delta_\gamma(r;L',A',\lambda',\kappa')$.
\end{lemma}
\begin{proof}
Take $0<M<1$ small enough that $\frac{\lambda}{L}-\frac{\lambda'}{L'}M>0$.
Then for any sufficiently large $C\geqslant 0$ the affine function $\theta\from r\mapsto Mr-C$ satisfies: 
\[\lambda'((L')^{-1}\theta(r)-A')-2\kappa' \leqslant \lambda(L^{-1}r-A)-\kappa\]

Fix $s\in \R$ and let $P$ be any path from $\gamma(s-r)$ to
$\gamma(s+r)$ that is disjoint from the ball of radius
$\lambda(L^{-1}r - A) -\kappa$ centered at $\gamma(s)$. 
By the above inequality it is also disjoint from the ball of radius
$\lambda'((L')^{-1}\theta(r)-A')-2\kappa'$ about $\gamma(s)$.

Let $\{x_0,x_1,\dots,x_l\}$ be the set $[s-r,s-\theta(r)]\cap(\Z\cup\{s-r,s-\theta(r)\})$ in descending order and let $P_-$ be the path from $\gamma(s-\theta(r))$ to $\gamma(s-r)$ obtained by concatenating geodesics $[\gamma(x_i),\gamma(x_{i+1})]$. 
Define a path $P_+$ from $\gamma(s+r)$ to $\gamma(s+\theta(r))$ similarly. 
Since $\kappa'\geqslant L'+A'$, the paths $P_-$ and $P_+$ are disjoint from the ball of radius $\lambda'((L')^{-1}\theta(r)-A')-\kappa'$ centered at $\gamma(s)$.

Define $P'$ to be the path from $\gamma(s-\theta(r))$ to $\gamma(s+\theta(r))$ obtained by concatenating $P_-$, $P$, and $P_+$.

Now, for each $r$ choose $s$ and $P$ so that $|P|\leqslant 1+\Delta_\gamma(r;L,A,\lambda,\kappa)$.
Then $\Delta_\gamma(\theta(r);L',A',\lambda',\kappa')\leqslant |P| +
2(L(r-\theta(r))+A)$. 
Since $\gamma$ is quasi-geodesic, $r\leqslant L|P|+LA$, so the
right-hand side can be bounded by an affine function of $\Delta_\gamma(r;L,A,\lambda,\kappa)$.
This proves one direction of the equivalence. The other follows immediately by reversing the roles in the above argument.
\end{proof}

We first give an example of the relationship between divergence and
contraction.
\begin{example}\label{ex:optimal}
Let $f(r)\geqslant r$ be an increasing, invertible function.
  Consider the space $X$ constructed in \fullref{ex:necklace}, but
  this time take $|I_i|:=2i$ and $|J_i|:=f(i)$ for $i\in\N$.
Let $\gamma$ be a geodesic whose image is $Y$.
Then $\Lambda_\gamma(i,\gamma^{-1}(y_i);1,0,1,1)=f(i)$, and this is
optimal for radius $i$, so $\Delta_\gamma\asymp f$.
On the other hand, the computation of \fullref{ex:necklace} shows that
$\diam\pi_Y^0(x_i)=2f^{-1}(4r)$. 
Thus, $\pi_Y^0$ is sublinearly contracting if and only if $f^{-1}$ is
sublinear, and, in this case, it is $(r,\rho)$--contracting for
$\rho\asymp f^{-1}$.
\end{example}

Our next proposition proves the implication $\eqref{item:csld}\implies\eqref{item:morse_qgeodesic}$ of \fullref{theorem:morse_equiv_csld}. 
It also gives a quantitave link between high divergence and contraction.

\begin{definition}
We say a function $g$ is \emph{completely super--}$f$ if for every
choice of $C_1>0$, $C_2>0$, $C_3\geqslant 0$, and $C_4\geqslant 0$ the collection of $r\in[0,\infty)$ such that $g(r)\leqslant C_1f(C_2r+C_3)+C_4$ is bounded. 
\end{definition}

\begin{proposition}\label{prop:csld_implies_sublinear_contraction}
Let $\gamma$ be a quasi-geodesic in a geodesic metric space $X$.
Suppose the empty set is not in the image of $\pi_\gamma^\almost$.
Let $f(r)\geqslant r$ be an increasing, invertible function.
If $\gamma$ has completely super--$f$ divergence, then there exists a function $\rho$ such that $\pi_\gamma^\almost$ is $(r,\rho)$--contracting and $\lim_{r\to\infty}\frac{\rho(r)}{f^{-1}(r)}=0$.

In particular, if $\gamma$ has completely superlinear divergence then there exists a sublinear function $\rho$ such that $\pi_\gamma^\almost$ is $(r,\rho)$--contracting.
\end{proposition}
\begin{proof} 
Let $\gamma$ be an $(L,A)$-quasi-geodesic.
Define:
\[\rho(r):=\sup_{d(x,y)\leqslant d(x,\gamma)\leqslant r}\diam\pi_\gamma^\almost(x)\cup\pi_\gamma^\almost(y)\]
To see that $\pi_\gamma^\almost$ is $(r,\rho)$--contracting we must show that $\rho$ is sublinear. 
Since $f(r)\geqslant r$, it suffices to prove the second claim: \[\lim_{r\to\infty}\frac{\rho(r)}{f^{-1}(r)}= 0\]

Suppose for a contradiction that $\limsup_{r\to\infty}\frac{\rho(r)}{f^{-1}(r)}>0$.
Then there exist $c>0$; sequences $(x_n)$ and $(y_n)$ with $x_n,\,
y_n\in X$,  $d(x_n,\gamma)\geqslant n$, 
and $d(x_n,y_n)\leqslant d(x_n,\gamma)$; and $x'_n\in\pi_\gamma^\almost(x_n)$ and $y'_n\in\pi_\gamma^\almost(y_n)$ such that: 
\begin{equation}
  \label{eq:8}
c f^{-1}(d(x_n,\gamma))\leqslant d(x'_n,y'_n)  
\end{equation}

Let $a_n$ and $b_n$ be such that $\gamma(a_n-b_n)=x'_n$ and
$\gamma(a_n+b_n)=y'_n$.
Define $m_n:=\gamma(a_n)$ and $R_n:=\frac{b_n}{L}-A$.
Since $\gamma$ is an $(L,A)$--quasi-geodesic,
$d(m_n,\{x'_n,y'_n\})\geqslant R_n$ and $b_n\geqslant \frac{d(x'_n,y'_n)-A}{2L}$.
By \eqref{eq:8} and the facts that $f^{-1}$ is unbounded and
increasing, $\lim_{n\to\infty} R_n=\infty$.

Choose $0<\lambda<\frac{1}{4}$ and $\kappa:=L+A$.

If there is a geodesic from $x_n$ to $y_n$ containing a point
$z$ such that $d(z,m_n)\leqslant \lambda R_n$, then:
\begin{align*}
  R_n&\leqslant d(y'_n,m_n)\\
&\leqslant d(y'_n,y_n) +d(y_n,z)+d(z,m_n)\\
&\leqslant d(y_n,\gamma)+\almost +d(y_n,z)+d(z,m_n)\\
&\leqslant \almost +2(d(y_n,z)+d(z,m_n))\\
&\leqslant \almost +2\lambda R_n+2d(y_n,z)\\
&= \almost
  +2\lambda R_n+2(d(x_n,y_n)-d(z,x_n))\\
&\leqslant \almost +2\lambda
  R_n+2(d(x_n,\gamma)-(d(x_n,\gamma)-\lambda R_n))\\
&= \almost +4\lambda R_n
\end{align*}
Thus, $R_n\leqslant \frac{\almost}{1-4\lambda}$.

If there is a geodesic from $x_n$ to $x_n'$ or from $y_n$ to $y_n'$
containing a point $z$ such that $d(z,m_n)\leqslant \lambda R_n$, then
a similar argument shows $R_n\leqslant \frac{\almost}{1-2\lambda}$.

Since $R_n\to\infty$, for all sufficiently large $n$ and
any choice of path $p_n$ that is a concatenation of geodesics
$[x'_n,x_n]$, $[x_n,y_n]$, $[y_n,y'_n]$, the path $p_n$ remains
outside the ball of radius $\lambda R_n$ about $m_n$.
This gives us a path of length at most $4d(x_n,\gamma)+2\almost$ from $\gamma(a_n-b_n)$ to $\gamma(a_n+b_n)$ that
remains outside the ball of radius $\lambda
\left(\frac{b_n}{L}-A\right)$ about $\gamma(a_n)$.

On the other hand, \eqref{eq:8} implies:
\[d(x_n,\gamma)\leqslant f\left(\frac{1}{c}d(x_n',y_n')\right)\leqslant f\left(\frac{2b_nL+A}{c}\right)\]
 We conclude that for all sufficiently large $n$ the $(L,A,\lambda,\kappa)$--divergence of $\gamma$
 evaluated at $b_n$ is at most $2\almost+4
 f\left(\frac{2b_nL+A}{c}\right)$, which contradicts the hypothesis
 that the divergence is completely super--$f$.
\end{proof}

The previous result can be strengthened to the statement:
\begin{proposition}\label{prop:stengthened}
Let $f$ be an increasing, invertible, completely superlinear function
satisfying the following additional condition:
\begin{equation}
  \label{eq:9}
\parbox{0.9\textwidth}{
For every $C$ there exists some $D$ such that 
for all $r>1$ and $k>D$ we have  $f(kr)>Cf(Cr+C)+C$.}  \tag{$\ast$} 
\end{equation}

If the divergence of $\gamma$ is at least $f$ then $\gamma$ is
$(r,\rho)$--contracting for some function  $\rho\preceq f^{-1}$.
\end{proposition}
\begin{proof}
For a contradiction we suppose that $\rho\not\preceq f^{-1}$ and replace
\eqref{eq:8} with $d(x'_n,y'_n)\geqslant n f^{-1}(d(x_n,\gamma))$. 
Using the
same method as in the proof of \fullref{prop:csld_implies_sublinear_contraction}, we deduce that for all sufficiently large $n$ the $(L,A,\lambda,\kappa)$--divergence of $\gamma$
 evaluated at $b_n$ is at most $2\almost+4
 f\left(\frac{2b_nL+A}{n}\right)$. 
Thus, $f(b_n)\leqslant 2\almost+4
 f\left(\frac{2b_nL+A}{n}\right)$.
Let $c_n:=b_n/n$ and $M:=\max\{2\almost, 4, 2L, A\}$.
Then, since $f$ is increasing:
\begin{equation}
  \label{eq:10}
  f(nc_n)\leqslant Mf(Mc_n+M)+M
\end{equation}
The left-hand side is unbounded as $n$ grows, so we immediately obtain
a contradiction if the sequence $(c_n)_{n\in\N}$ is bounded.
If the sequence is unbounded then, by passing to a subsequence, we may
assume $c_n>1$ for all $n$. 
In this case the inequality \eqref{eq:10} holds for all $n$, which contradicts
condition  \eqref{eq:9}.
\end{proof}
Suitable functions $f$ for \fullref{prop:stengthened} include $f(r):=r^d$, $r^d/\log(r)$, $r\log(r)$
and $d^r$ for any $d>1$. 
The function $f(r):=2^{2^{2^{1+\lfloor \log_2\log_2 r \rfloor}}}$ is
completely superlinear, but does not satisfy \eqref{eq:9}, since
$f(n2^{2^{n-1}})=f(2^{2^{n-1}})$ for all $n\in\N$.

\begin{corollary} If a quasi-geodesic $\gamma$ has divergence at least $r^k$ then $\gamma$ is $(r,r^{1/k})$--contracting. 
If it has exponential divergence, then $\gamma$ is logarithmically contracting.
Finally, if it has infinite divergence, then it is strongly contracting.
\end{corollary}

Here infinite divergence means $\Delta_\gamma(r)=\infty$ for all $r$ large enough. \fullref{ex:optimal} shows these conclusions are optimal.

\bigskip

We now address the implication
$\eqref{item:morse_qgeodesic}\implies\eqref{item:csld}$ of
\fullref{theorem:morse_equiv_csld}. 
 In this direction we can show that the Morse property implies completely superlinear
 divergence, but we do not get explicit control of the divergence
 function in terms of the Morse function, see
 \fullref{lem:morseimpliescsld}.

There is one special case in which we can say more. 
Charney and Sultan \cite{ChaSul15} recently gave a proof\footnote{The
  original proof of this fact is due to Behrstock and Dru{\c{t}}u \cite{BehDru14}, by
  different methods.} that if $\alpha$ is a Morse
geodesic in a CAT(0) space then $\alpha$ has at least quadratic
divergence. 
Essentially the same argument gives a general result:
\begin{proposition}
  Let $\alpha$ be a geodesic in a geodesic metric space $X$. 
If $\alpha$ is $(\rho_1,\rho_2)$--contracting with $\rho_2$ bounded, then $\Delta_\alpha(r)\asympgeq r\rho_1(r)$.
\end{proposition}

\begin{lem}\label{lem:constructingquasigeods}
Let $X$ be a geodesic metric space.
Let $a,b,c,d\in X$ and $r>0$ satisfy the following conditions:
\begin{enumerate}
 \item $d(a,d)\geqslant r$
 \item There exists a path $\gamma$ from $a$ to $d$ passing through
   $b$ and $c$ such that the length of $\gamma$ is at most $Cr$ 
   and such that $[a,b]_\gamma$, $[b,c]_\gamma$,
   and $[c,d]_\gamma$ are continuous $(L,0)$--quasi-geodesics.
 \item The path $\gamma$ does not contain a point within distance $\lambda r$
   of $e$, where $e$ is the midpoint of a geodesic from $a$ to $d$.
\end{enumerate}
For any $L'>\max\{L,C,C/\lambda\}\geqslant 1$ there exists a continuous $(L',0)$--quasi-geodesic $\gamma'$ from $a$ to $d$ of length at most $|\gamma|$ such that $\gamma'$ does not contain a point within distance $\lambda r/2$ of $e$.
\end{lem}
\begin{proof}
The construction of $\gamma'$ is exactly as in \fullref{prop:morseimpliessublinearcontraction} with $L$ replaced by $L'$.
This involves finding points $p$ and $q$ on $\gamma$ such that
$L'd(p,q)=|[p,q]_\gamma|$ and replacing $[p,q]_\gamma$ by a geodesic
with the same endpoints. 
Now, $d(p,q)\leqslant |\gamma|/L'<\lambda r$, so for any point $z$ on
a newly introduced geodesic segment we have $d(z,e)\geqslant
d(\gamma,e)-d(p,q)/2>\lambda r/2$.
\end{proof}

\begin{proposition}\label{lem:morseimpliescsld} 
Let $\gamma$ be a Morse quasi-geodesic in a geodesic metric space $X$. Then the divergence of $\gamma$ is completely superlinear.
\end{proposition}
\begin{proof} We prove the contrapositive. 
Let $\gamma$ be an $(L,A)$--quasi-geodesic and suppose its divergence is not completely superlinear.
Then there exists $C>0$ for which there exists an unbounded sequence of numbers $r_n\geqslant 1$ and paths $p_n$ such that:
\begin{enumerate}
\item There exists a sequence of real numbers $s_n$ such that the endpoints of $p_n$ are $x_n=\gamma(s_n-r_n)$ and $y_n=\gamma(s_n+r_n)$.
\item $|p_n|\leqslant Cr_n$.
\item $p_n$ does not intersect the $(\frac{r_n}{2L}-A)$--neighborhood of $\gamma(s_n)$.\label{item:outside}
\end{enumerate}

We may assume all $r_n\geqslant 4AL$ so point \eqref{item:outside} can be
replaced by:
\begin{enumerate}
\item[\ref{item:outside}$'$\!.] $p_n$ does not intersect the $(\frac{r_n}{4L})$--neighborhood of $m_n:=\gamma(s_n)$.
\end{enumerate}

Our goal is to construct uniform quasi-geodesics $\gamma_n$ from $x_n$ to $y_n$ that avoid increasingly large balls around $m_n$.

Set $x_{n,0}:=x_n$ and define $x_{n,1}$ to be the last point on $p_n$
for which we have $d(x_{n,0},x_{n,1})=r_n/8L$.

Similarly define $x_{n,i}$ to be $y_n$ if $d(x_{n,i-1},y_n)< r_n/4L$ or to be the last point on $p_n$ satisfying $d(x_{n,i-1},x_{n,i})=r_n/8L$ otherwise.

Note that $y_n=x_{n,k_n}$ for some $k_n\leqslant 8CL$. By construction, if $i\neq j$ then $d(x_{n,i},x_{n,j})\geqslant r_n/8L$.

Let $\gamma_n^1$ be a concatenation of geodesics
$[x_{n,0},x_{n,1}]\dots[x_{n,k_n-1},y_n]$. 
We have that $|\gamma_n^1|\leqslant Cr_n$ and $d(\gamma_n^1,m_n)>r_n/8L$.

Applying \fullref{lem:constructingquasigeods} for each $1\leqslant
i\leqslant\lfloor k_n/3\rfloor$  there are
$(L_2,0)$--quasi-geodesics (where $L_2$ does not depend on $n$) 
from $x_{n,3(i-1)}$ to $x_{n,3i}$ such that the concatenation $\gamma_n^2$ of these with $[x_{n,3\lfloor k_n/3\rfloor},y_n]_{\gamma_n^1}$ satisfies $d(\gamma_n^2,m_n)>r_n/16L$.

Repeating this procedure at most $d=\lceil\log_3 8CL \rceil$ times we
obtain an $(L_d,0)$ quasi-geodesic $\gamma_n^d$ from $x_n$ to $y_n$
satisfying  $d(\gamma_n^d,m_n)>r_n/(2^{d+2}L)$. 
Again, $L_d$ does not depend on $n$.

If $\gamma$ is $\mu$--Morse, then the $\gamma_n^d$ are $\mu'$--Morse for
some $\mu'$ that does not depend on $n$. 
Then $d(\gamma_n^d,m_n)\leqslant \mu'(K,C)$, which is bounded, contradicting the lower bound above.
\end{proof}

A finitely generated group is called \emph{constricted} if all of its
asymptotic cones have cut points \cite{DruSap05}.

\begin{cor}\label{cor:divconecutpts} 
Suppose there exists a quasi-geodesic $\gamma$  with completely super\-linear divergence in a geodesic metric space $X$.
In every asymptotic cone of $X$ every point of the ultralimit of $\gamma$ is a cut point.

In particular, a finitely generated group is constricted if one of
its Cayley graphs contains a quasi-geodesic with completely superlinear divergence.
\end{cor}

Olshanskii, Osin, and Sapir \cite[Corollary~$6.4$]{OlsOsiSap09} build
a group that has an asymptotic cone with no cut point such that the
group has a Cayley graph with geodesics of superlinear divergence. 
These geodesics are therefore not Morse.
They explicitly state that their construction yields geodesics that are not completely superlinear. 
\fullref{cor:divconecutpts} shows that this will be the case in any such construction.

\section{Robustness}\label{sec:robust}
Suppose that $\pi_Y^\almost$ is $(\rho_1,\rho_2)$--contracting.
In this section we investigate the extent to which $\rho_2$ is
affected by changes to $\rho_1$, $\almost$, or $Y$.

Clearly $\pi_Y^\almost$ is $(\rho_1',\rho_2)$--contracting for
$\rho_1'\leqslant \rho_1$.
From \fullref{theorem:morse_equiv_contracting} we know that
$\pi_Y^\almost$ is $(r,\rho_2')$--contracting for some $\rho_2'$
depending on $\rho_1$ and $\rho_2$. 
For this $\rho_2'$, it follows that $\pi_Y^\almost$ is $(\rho_1',\rho_2')$--contracting
for every $\rho_1\leqslant\rho_1'\leqslant r$.

In general $\rho_2$ and $\rho_2'$ are not asymptotic.
For example, if $\pi_Y^\almost$ is $(r/2,B_1)$--contracting it
is $(r,\rho_2)$--contracting for $\rho_2\asymp \log_2$, as in
\fullref{prop:Abel}, but not necessarily $(r,B_2)$--contracting for
some constant $B_2$, by \fullref{ex:log}.
One well-known special case is that $(r/M,B_1)$--contracting
for $M>1$ and $B_1$ bounded implies $(r/2,B_2)$--contracting
for some bounded $B_2$, see, eg, \cite{Sul12}.

The output contraction functions are asymptotic when the input
function is changed by an additive constant:
\begin{lemma}\label{lem:removeconstantinsublinearcontraction}
  If $\pi^\almost_Y$ is $(\rho_1,\rho_2)$--contracting for
  $\rho_1(r)=\rho_1'(r)-C$, with $\rho_1'(r)\leqslant r$ and
  $C\geqslant 0$,  then $\pi_Y^\almost$ is
  $(\rho_1',\rho_2')$--contracting for some
  $\rho_2'\asymp\rho_2$. 
\end{lemma}
\begin{proof}
  Let $C':=\sup\{r\mid \rho_1(r)\leqslant C\}$.
Suppose that $x$ and $y$ are points with $d(x,y)\leqslant
\rho_1'(d(x,Y))$.
If $d(x,y)\leqslant \rho_1(d(x,Y))=\rho_1'(d(x,Y))-C$ then we have
$\diam\pi^\almost_Y(x)\cup\pi^\almost_Y(y)\leqslant \rho_2(d(x,Y))$.
Otherwise, let $z$ be a point on a geodesic from $x$ to $y$ such that
$d(x,z)=\rho_1(d(x,Y))$.
This implies $d(y,z)\leqslant C$.
Now: 
\begin{align*}
  \diam\pi^\almost_Y(x)\cup\pi^\almost_Y(y)&\leqslant\diam\pi^\almost_Y(x)\cup\pi^\almost_Y(z)+\diam\pi^\almost_Y(z)\cup\pi^\almost_Y(y)\\
&\leqslant \rho_2(d(x,Y))+\diam\pi^\almost_Y(z)\cup\pi^\almost_Y(y)
\end{align*}

If $d(z,y)>\rho_1(d(z,Y))$ then $d(z,Y)\leqslant C'$, so
$\diam\pi^\almost_Y(z)\cup\pi^\almost_Y(y)\leqslant 2(C+C'+\almost)$.
If $d(z,y)\leqslant\rho_1(d(z,Y))$ then
$\diam\pi^\almost_Y(z)\cup\pi^\almost_Y(y)\leqslant\rho_2(d(z,Y))\leqslant\rho_2(2d(x,Y))$.
Combining these cases, we see that $d(x,y)\leqslant \rho_1(d(x,Y))$ implies:
\[\diam\pi^\almost_Y(x)\cup\pi^\almost_Y(y)\leqslant\rho_2(d(x,Y))+\rho_2(2d(x,Y))+2(C+C'+\almost)\]
Thus, it suffices to take $\rho_2'(r):=2\rho_2(2r)+2(C+C'+\almost)$.
\end{proof}

Next, consider changes to the projection parameter.
\begin{lemma}\label{lem:almost}
Suppose $\almost_0$ and $\almost_1$ are constants such that the
empty set is neither in the image of $\pi_Y^{\almost_0}\from X\to
2^Y$ nor in the image of $\pi_Y^{\almost_1}\from X\to 2^Y$.
If $\pi_Y^{\almost_0}$ is $(\rho_1,\rho_2)$--contracting then there
exist $\rho_1'$ and $\rho_2'$ such that $\pi_Y^{\almost_1}$ is
$(\rho_1',\rho_2')$--contracting.
If $\almost_1\leqslant\almost_0$ or if $\rho_1(r):=r$ then we can take $\rho_1'=\rho_1$ and $\rho_2'\asymp\rho_2$.
\end{lemma}
\begin{proof}
When $\almost_1\leqslant\almost_0$ we have
$\pi_Y^{\almost_1}(x)\subset\pi_Y^{\almost_0}(x)$, so the result is clear.
 In this case $\rho'_1=\rho_1$ and $\rho_2'=\rho_2$ will suffice.

The fact that $\pi_Y^{\almost_1}$ is sublinearly contracting follows
from \fullref{theorem:morse_equiv_contracting}, since $Y$ is Morse.
It remains only to prove the asymptotic statement in the case that
$\rho_1(r):=r$, so suppose $\pi_Y^{\almost_0}$ is $(r,\rho_2)$--contracting.

For any $x\in X\setminus Y$ and each $i\in\{0,1\}$,
consider a point $x_i\in\pi_Y^{\almost_i}(x)$ and a point $z_i$ on a
geodesic from $x$ to $x_i$ with $d(x,z_i)=d(x,Y)$. 
Then:
\begin{align*}
  d(x_0,x_1)&\leqslant d(x_0,z_0)+d(z_0,\pi^{\almost_0}_Y(z_0))+\diam
              \pi_Y^{\almost_0}(z_0)\cup \pi_Y^{\almost_0}(x)\\
&\qquad+\diam
              \pi_Y^{\almost_0}(x)\cup\pi_Y^{\almost_0}(z_1)+d(\pi_Y^{\almost_0}(z_1),z_1)+d(z_1,x_1)\\
&\leqslant
  \almost_0+2\almost_0+\rho_2(d(x,Y))+\rho_2(d(x,Y))+\almost_0+\almost_1+\almost_1\\
&=4\almost_0+2\almost_1+2\rho_2(d(x,Y))
\end{align*}

If $d(x,y)\leqslant d(x,Y)$ then:
\begin{align*}
\diam \pi_Y^{\almost_1}(x)\cup\pi_Y^{\almost_1}(y)&\leqslant \diam\pi_Y^{\almost_1}(x)\cup\pi_Y^{\almost_0}(x)+\diam 
\pi_Y^{\almost_0}(x)\cup\pi_Y^{\almost_0}(y)\\
 &\qquad+\diam \pi_Y^{\almost_0}(y)\cup\pi_Y^{\almost_1}(y)\\
&\leqslant  4\almost_0+2\almost_1+2\rho_2(d(x,Y))      +
  \rho_2(d(x,Y))\\
&\qquad +4\almost_0+2\almost_1+2\rho_2(d(y,Y))
\end{align*}

Since $d(y,Y)\leqslant 2d(x,Y)$, this means that
 $\pi_Y^{\almost_1}$ is $(r,\rho_2')$--contracting for:
 \[\rho_2'(r):=8\almost_0+4\almost_1+3\rho_2(r)+2\rho_2(2r)\asymp \rho_2(r)\qedhere\]
\end{proof}

Finally, consider changes to the target of the projection map.
\begin{lem}\label{lem:strcontbddHdist}
Let $Y$ and $Y'$ be subspaces of a geodesic metric space $X$ at bounded Hausdorff distance from
one another. 
Suppose that $\pi_Y^\almost$ is $(\rho_1,\rho_2)$--contracting.
Then $\pi_{Y'}^{\almost}$ is $(r,\rho_2')$--contracting for some $\rho_2'$.
If $\rho_1(r)=r$ then we can take $\rho_2'\asymp\rho_2$.
\end{lem}
\begin{proof}
Let $C$ be the Hausdorff distance between $Y$ and $Y'$.

For every $x\in X$ we have $\pi^{\almost}_{Y'}(x)\subset
\overline{N}_C(\pi^{\almost+2C}_Y(x))$.
The result now follows easily from \fullref{lem:almost}.
\end{proof}

In light of \fullref{lem:almost}, we can speak of the set $Y$ being a
contracting set if some $\almost$--closest point projection to $Y$ is contracting.

\begin{definition}\label{def:spacecontracting}
  We say \emph{$Y$ is $(\rho_1,\rho_2)$--contracting} if there exists
  an $\almost\geqslant 0$ such that the $\almost$--closest point projection
  $\pi_Y^\almost\from X\to 2^Y$ is $(\rho_1,\rho_2)$--contracting.

Equivalently, \emph{$Y$ is $(\rho_1,\rho_2)$--contracting} if for all sufficiently small $\almost\geqslant 0$, if
$\pi_Y^\almost$ does not have the empty set in its image, then $\pi_Y^\almost$ is
$(\rho_1,\rho_2)$--contracting. 
\end{definition}

\section{Geodesic image theorem}
In this section we give an additional characterization of sublinear
contraction in terms of projections of geodesic segments.
\begin{theorem}\label{thm:git}
Let $Y$ be a subspace of a geodesic metric space $X$.
Suppose the empty set is not in the image of $\pi_Y^\almost$.
  The following are equivalent:
  \begin{enumerate}
  \item There exist a sublinear function $\rho$ and a constant
    $C\geqslant 0$
    such that for every
geodesic segment $\gamma\subset X$, with endpoints denoted $x$ and $y$, if $d(\gamma,Y)\geqslant
 C$ then
 $\diam\pi_Y^\almost(\gamma)\leqslant\rho(\max\{d(x,Y),d(y,Y)\})$.\label{item:endpointbound}
 \item There exist a sublinear function $\rho'$ and a constant
    $C'\geqslant 0$
    such that for every
geodesic segment $\gamma\subset X$, if $d(\gamma,Y)\geqslant
 C'$ then $\diam\pi_Y^\almost(\gamma)\leqslant\rho'(\max_{z\in\gamma}d(z,Y))$.\label{item:maxbound}
\item There exists a sublinear function $\rho''$ such that $\pi_Y^\almost$ is 
  $(r,\rho'')$--contracting. \label{item:sublinear_contraction}  
\end{enumerate}
Moreover, $\rho\asymp\rho'\asymp\rho''$.
\end{theorem}
See \fullref{fig:log}, letting $\gamma$ be a
subsegment of $\cup_i W_i$.

The case that $Y$ is strongly contracting, that is, 
$\rho''$ is bounded, recovers the well-known `Bounded Geodesic Image
Property', cf \cite{MasMin00, BesFuj09}.
\begin{corollary}\label{corollary:bgi}
  If $Y$ is strongly contracting, $R_2\geqslant 1$ is a constant greater than
  twice the bound on the contraction function for $Y$, and $\gamma$ is a geodesic segment
  that does not enter the $R_2$--neighborhood of $Y$ then
  $\diam\pi_Y^\almost(\gamma)$ is bounded, with bound depending only
  on $\almost$ and $\rho''$.
\end{corollary}

Alternatively, one could read \fullref{thm:git} as saying that if
$\pi_Y^\almost$ is sublinearly contracting and $\gamma$ is a geodesic
ray that is far from $Y$, but such that $\pi_Y^\almost(\gamma)$ is
large, then $d(\gamma(t),Y)$ grows superlinearly with respect to
$\diam \pi_Y^\almost(\gamma([0,t]))$. 

\begin{proof}[{Proof of \fullref{thm:git}}]\mbox{}\\
{\itshape\eqref{item:endpointbound}$\implies$\eqref{item:sublinear_contraction}:}
 Define $\rho_1(r):=r-C$ and $\rho_2(r)=\rho(2r-C)$.
By \fullref{lem:removeconstantinsublinearcontraction}, it suffices to
show that $\pi_Y^\almost$ is $(\rho_1,\rho_2)$--contracting.

Suppose $x$ and $y$ are points of $X$ with $d(x,y)\leqslant
\rho_1(d(x,Y))$, and let $\gamma$ be a geodesic from $x$ to $y$.
Then $\gamma$ remains outside the $C$--neighborhood of $Y$, by the
definition of $\rho_1$, so:
\begin{align*}
  \diam\pi_Y^\almost(x)\cup\pi_Y^\almost(y)&\leqslant\diam\pi_Y^\almost(\gamma)\\
&\leqslant\rho(\max\{d(x,Y),d(y,Y)\})\\
&\leqslant \rho(2d(x,Y)-C)=\rho_2(d(x,Y))
\end{align*}
This proves
{\itshape\eqref{item:endpointbound}$\implies$\eqref{item:sublinear_contraction}},
and a similar argument proves {\itshape\eqref{item:maxbound}$\implies$\eqref{item:sublinear_contraction}}.

Now assume {\itshape \eqref{item:sublinear_contraction}}.
If $d(x,y)\leqslant d(x,Y)+d(y,Y)$ then both
{\itshape\eqref{item:endpointbound}} and
{\itshape\eqref{item:maxbound}} follow easily, so assume not.
Let $z_0$ be the point of $\gamma$ at distance $d(x,Y)$ from $x$.
Our assumption says $d(z_0,y)>d(y,Y)$.
Define points $z_{i+1}$ inductively as follows: if $d(z_i,y)>d(y,Y)+d(z_i,Y)$
define $z_{i+1}$ to be the point of $\gamma$ between $z_i$ and $y$ at
distance $d(z_i,Y)$ from $z_i$.
Let $k$ be the last index so defined.
From these choices we estimate:
\begin{align}
  \diam \pi_Y^\almost(\gamma)&\leqslant \diam
  \pi_Y^\almost(\overline{N}_{d(x,Y)}(x))+\sum_{i=0}^{k}\diam
                               \pi_Y^\almost(\overline{N}_{d(z_i,Y)}(z_i))\notag\\
&\qquad\qquad +\diam\pi_Y^\almost(\overline{N}_{d(y,Y)}(y))\notag\\
&\leqslant 2\left(\rho''(d(x,Y))+\sum_{i=0}^{k}\rho''(d(z_i,Y)) +\rho''(d(y,Y))\right)\label{projectionestimate}
\end{align}

Since $\gamma$ is a geodesic:
\begin{align}
  d(x,y)&=d(x,z_0)+\sum_{i=0}^{k-1}d(z_i,z_{i+1}) + d(z_k,y)\notag\\
&=d(x,Y)+\sum_{i=0}^{k-1}d(z_i,Y) +d(z_k,y)\label{subsegments}
\end{align}

We can also bound $d(x,y)$ in terms of the projections to $Y$:
\begin{align}
d(x,y)&\leqslant d(x,\pi_Y^\almost(x))+\diam\pi_Y^\almost(x)\cup\pi_Y^\almost(y)+d(\pi_Y^\almost(y),y)\notag\\
  &\leqslant d(x,\pi_Y^\almost(x))+\diam\pi_Y^\almost(x)\cup\pi_Y^\almost(z_0)+\sum_{i=0}^{k-1}\diam\pi_Y^\almost(z_i)\cup\pi_Y^\almost(z_{i+1})\notag\\
&\qquad\qquad+\diam\pi_Y^\almost(z_k)\cup\pi_Y^\almost(y)+d(\pi_Y^\almost(y),y)\notag\\
&\leqslant d(x,Y)+\almost+\rho''(d(x,Y))+\sum_{i=0}^{k-1}\rho''(d(z_i,Y))\label{detour}\\
&\qquad\qquad+\rho''(d(z_k,Y))+\rho''(d(y,Y))+d(y,Y)+\almost\notag
\end{align}

Combining \eqref{subsegments} and \eqref{detour} gives us the
estimate:
\begin{multline}
  \label{eq:11}
  \sum_{i=0}^{k-1}d(z_i,Y)-\rho''(d(z_i,Y))\leqslant\\ 2\almost+\rho''(d(x,Y))+\rho''(d(z_k,Y))+\rho''(d(y,Y))+d(y,Y)-d(z_k,y)
\end{multline}

Define $R_n\geqslant 0$ such that for all $r\geqslant R_n$ we have $0\leqslant
\rho''(r)\leqslant r/n$. 
Suppose that $d(\gamma,Y)\geqslant R_2$ so that
$d(z_i,Y)-\rho''(d(z_i,Y))\geqslant \rho''(d(z_i,Y))$ for all $i$.
These bounds, along with \eqref{eq:11}, \eqref{projectionestimate}, and $E:=d(z_k,y)-d(y,Y)$
give:
\[\diam \pi_Y^\almost(\gamma)\leqslant 2\left(2\big(\almost+\rho''(d(x,Y))+ \rho''(d(z_k,Y))+\rho''(d(y,Y))\big)-E\right)
\]
By construction, $E>0$, so to prove {\itshape
  \eqref{item:maxbound}} it suffices to take $C':=R_2$ and
$\rho(r):=4\almost +12\rho''(r)$.

To prove {\itshape
  \eqref{item:endpointbound}} we suppose  $d(\gamma,Y)\geqslant C:=R_4\geqslant R_2$ and bound $2
\rho''(d(z_k,Y))-E$ in terms of $\rho''(d(y,Y))$.
There are two cases to consider.
If $d(z_k,Y)\leqslant 4d(y,Y)$ then $2\rho''(d(z_k,Y))-E\leqslant
2\rho''(4d(y,Y))$.
Otherwise, $d(z_k,Y)>4d(y,Y)$ implies $E>d(z_k,Y)/2$, so:
\[2\rho''(d(z_k,Y))-E<2\frac{d(z_k,Y)}{4}-\frac{d(z_k,Y)}{2}=0\]
Thus, it suffices to take $\rho'(r):=4\almost+12\rho''(4r)$.
\end{proof}

\section{Further applications}
First, we prove a general result. 
\begin{proposition}\label{lemma:quasiconvex}
Let $X$ be a geodesic metric space.
Suppose subspaces $Y$ and $Y'$ of $X$ are $\mu$--Morse. 
Let $\almost\geqslant 0$ be a constant such that there exist points $p\in Y$ and
$p'\in Y'$ such that $d(p,p')\leqslant d(Y,Y')+\almost$.
Then there exist a constant $B$ and a sublinear function $\rho$, each
depending only on $\mu$ and $\almost$, satisfying the following conditions:
\begin{itemize}
\item If $d(Y,Y')\leqslant 2\mu(4,0)$ then $Y\cup Y'$ is $B$--quasi-convex.
\item If $d(Y,Y')>2\mu(4,0)$ then for every geodesic $\alpha$ from $Y$ to $Y'$ with
  $|\alpha|\leqslant d(Y,Y')+\almost$ and every geodesic $\gamma$ from
  $Y$ to $Y'$ we have $d(\alpha,\gamma)<\rho(d(Y,Y'))$.
\end{itemize}
\end{proposition}
\begin{proof}
Take geodesics $\alpha$ and $\gamma$ as hypothesized.
Let $\beta$ be a geodesic from $\alpha$ to $\gamma$ with
$|\beta|=d(\alpha,\gamma)$. See \fullref{fig:constriction}.
\begin{figure}[h]
  \centering
\labellist
\small
\pinlabel $Y$ [r] at 34 59
\pinlabel $Y'$ [l] at 285 55
\pinlabel $\alpha$ [t] at 122 13
\pinlabel $\gamma$ [bl] at 123 55
\pinlabel $\beta$ [r] at 160 34
\tiny
\pinlabel $p$ [r] at 41 12
\pinlabel $q$ [r] at 8 119
\pinlabel $q'$ [l] at 312 119
\pinlabel $p'$ [l] at 280 12
\pinlabel $u$ [t] at 81 11
\pinlabel $u'$ [t] at 239 12
\pinlabel $x$ [t] at 160 11
\pinlabel $y$ [b] at 160 52
\pinlabel $v$ [b] at 85 61
\pinlabel $v'$ [b] at 233 60
\endlabellist
  \includegraphics[scale=.9]{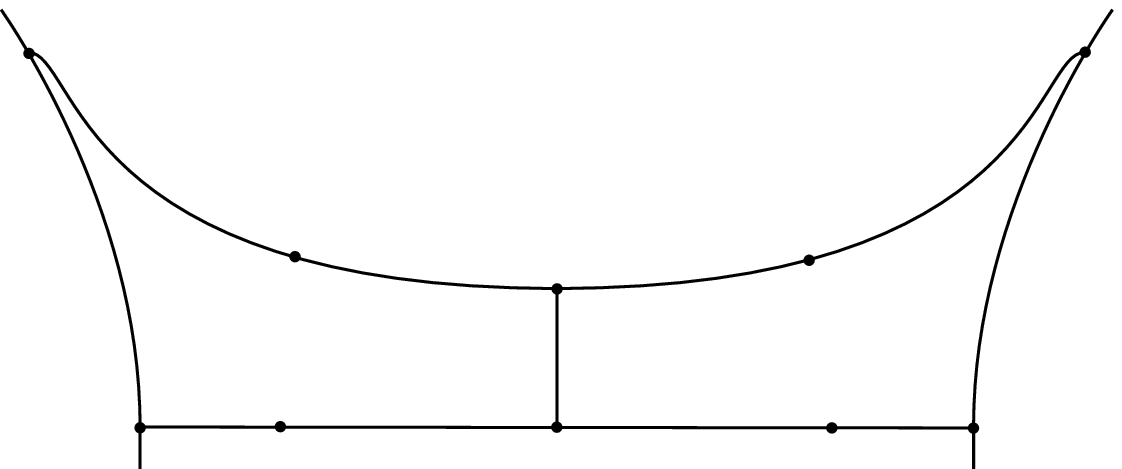}
  \caption{Setup for \fullref{lemma:quasiconvex}}
  \label{fig:constriction}
\end{figure}
Let $\delta:=[p,x]_\alpha\beta[y,q]_\gamma$ and
$\delta':=[p',x]_\alpha\beta[y,q']_\gamma$. (Recall that $[p,x]_\alpha$
  denotes the subsegment of $\alpha$ from $p$ to $x$.)
Suppose that $\delta$ fails to be a $(k,0)$--quasi-geodesic for some
$k>3$.
Both $[p,x]_\alpha\beta$ and $\beta[y,q]_\gamma$ are
$(3,0)$--quasi-geodesics, by minimality of $d(x,y)$, so there exist
points $u\in[p,x]_\alpha$ and $v\in[y,q]_\gamma$ such that $kd(u,v)<d(u,x)+d(x,y)+d(y,v)$.
Now, $d(v,y)\leqslant d(v,u)+d(u,x)+d(x,y)$, so:
\[(k-1)d(x,y)\leqslant (k-1)d(u,v)<2(d(u,x)+d(x,y))\]
Whence:
\begin{equation}
  \label{eq:7}
  d(\alpha,\gamma)=d(x,y)<\frac{2d(u,x)}{k-3}\leqslant\frac{2|\alpha|}{k-3}\leqslant\frac{2(d(Y,Y')+\almost)}{k-3}
\end{equation}

If $d(Y,Y')\leqslant 2\mu(4,0)$ and $\delta$ is not a
$(4,0)$--quasi-geodesic then $d(\alpha,\gamma)<4\mu(4,0)+2\almost$, by
\eqref{eq:7}.
This means $[y,q]_\gamma$ is a geodesic with one endpoint on $Y$ and
one within distance $6\mu(4,0)+2\almost$ of $Y$. 
Since $Y$ is $\mu$--Morse there is a $B_0$ depending on $\mu$ such that
such a geodesic segment is contained in the $B_0$--neighborhood of $Y$.

If $\delta$ is a $(4,0)$--quasi-geodesic it is contained in the
$\mu(4,0)$--neighborhood of $Y$.

The same arguments apply for $\delta'$, and $\gamma\subset
\delta\cup\delta'$, so if $d(Y,Y')\leqslant 2\mu(4,0)$ then $Y\cup Y'$
is $B$--quasi-convex for $B:=\max\{B_0,\mu(4,0)\}$.

Now suppose $d(Y,Y')>2\mu(4,0)$. Then $\delta$ and $\delta'$ cannot
both be $(4,0)$--quasi-geodesics. By \eqref{eq:7}:
  \begin{align*}
  d(\alpha,\gamma)&<\frac{2(d(Y,Y')+\almost)}{\sup\{k\in\R\mid
                    \delta\text{ or }\delta'\text{ is not a $(k,0)$--quasi-geodesic}\}-3}\\
&\leqslant\frac{2(d(Y,Y')+\almost)}{\sup\{k\in\R\mid  d(Y,Y')>2\mu(k,0)\}-3} 
  \end{align*}
Define: \[\rho(r):=\frac{2(r+\almost)}{\sup\{k\in\R\mid
  r>2\mu(k,0)\}-3}\]
We interpret $\rho(r)$ to be 0 if $\{2\mu(k,0)\}_{k\in\R}$ is bounded above by $r$.
For $r\geqslant \almost$ we have: 
\[\frac{\rho(r)}{r}\leqslant \frac{4}{\sup\{k\in\R\mid
  r>2\mu(k,0)\}-3}\]
The denominator is unbounded and non-decreasing as a function of $r$,
so we have $\lim_{r\to\infty}\frac{\rho(r)}{r}=0$.
\end{proof}

 We first give an application of the second part of \fullref{lemma:quasiconvex}.

\begin{proposition}\label{prop:sublinearprojection}
  Let $X$ be a geodesic metric space and let $Y$ and $Y'$ be
  $\mu$--Morse subspaces of $X$.
Let $\almost\geqslant 0$ be a constant such that 
the image of  $\pi_Y^\almost$ does not contain the
empty set, and such that there exist points $p\in Y$ and $p'\in Y'$
such that $d(p,p')\leqslant d(Y,Y')+\almost$.

Suppose $d(Y,Y')>2\mu(6,0)$. 
Then there is a sublinear function $\rho$ depending only on $\mu$ such that
$\diam\pi_Y^\almost(Y')\leqslant \rho(d(Y,Y'))$.
\end{proposition}
\begin{proof}
Since $Y$ is $\mu$--Morse, there is a sublinear function $\rho'$
depending only on $\mu$ such
that $Y$ is $(r,\rho')$--contracting, by
\fullref{prop:morseimpliessublinearcontraction}.

Note that $p\in\pi_Y^\almost(p')$.
Choose $q'\in Y'$ and $q\in\pi_Y^\almost(q')$. 
Let $\gamma$ be a geodesic from $q$ to $q'$, let $\alpha$ be a
geodesic from $p$ to $p'$, and let $x\in\alpha$ and $y\in\gamma$ be
points such that $d(x,y)=d(\alpha,\gamma)$.
The setup is the same as in \fullref{lemma:quasiconvex}, and we make
the corresponding definitions of $\delta$, $\delta'$, etc.

Suppose $\delta'$ is not a $(5,0)$--quasi-geodesic.
Define $u'$ and $v'$ as in \fullref{lemma:quasiconvex}, so that
$d(u',x)+d(x,y)+d(y,v')>5d(u',v')$.
We have $p\in\pi_Y^\almost(u')$ and $q\in\pi_Y^\almost(v')$.
By definition of $x$ and $y$, we know $d(x,y)\leqslant d(u',v')$, so $d(u',x)+d(y,v')>4d(u',v')$.
In particular, we have $2d(u',v')<d(u',x)$ or
$2d(u',v')<d(v',y)$.
We suppose the former, the other case being similar.

First, suppose that $d(u',Y)<\almost$. Then:
\begin{align*}
  d(p,q)&\leqslant d(p,v')+d(v',q)\\
&\leqslant 2d(p,v')+\almost\\
&\leqslant 2(d(p,u')+d(u',v'))+\almost\\
&\leqslant 2d(p,u')+d(u',x)+\almost\\
&\leqslant 3(d(u',Y)+\almost)+\almost<7\almost
\end{align*}

Otherwise, if $d(u',Y)\geqslant\almost$, then we have:
\[d(u',v')<\frac{1}{2}d(u',x)\leqslant\frac{1}{2}(d(u',Y)+\almost)\leqslant
d(u',Y)\]
By the contraction property:
\[d(p,q)\leqslant\diam\pi_Y^\almost(u')\cup\pi_Y^\almost(v')\leqslant\rho'(d(u',Y))\leqslant\rho'(d(Y,Y')+\almost)\]

Suppose instead that $\delta'$ is a $(5,0)$--quasi-geodesic. 
Then $\delta$ is not a $(6,0)$--quasi-geodesic, since $d(Y,Y')>2\mu(6,0)$.
 By \eqref{eq:7} we
have: 
\[d(x,y)<\frac{2}{3}(d(x,u))\leqslant
\frac{2}{3}(d(x,Y)+\almost)\]

If $d(x,Y)\leqslant 2\almost$
it follows that $d(x,y)\leqslant 2\almost$.
Thus $d(y,Y)\leqslant
d(y,x)+d(x,Y)\leqslant 4\almost$, and:
\[d(p,q)\leqslant  d(q,y) + d(y,x)+d(x,p)\leqslant d(y,Y)+\almost
+2\almost + d(x,Y)+\almost \leqslant 10\almost\]

Otherwise $d(x,Y)>2\almost$ and it follows that $d(x,y)\leqslant
d(x,Y)$. 
We then use the contraction property to see:
\[d(p,q)\leqslant\diam \pi_Y^\almost(x)\cup\pi_Y^\almost(y)\leqslant
\rho'(d(x,Y))\leqslant \rho'(d(Y,Y')+\almost)\]

Since $q'$ was an arbitrary point in $Y'$ and $q$ was an arbitrary
point of $\pi_Y^\almost(q')$, we conclude $\diam
\pi_Y^\almost(Y')\leqslant 2(\rho'(d(Y,Y')+\almost)+10\almost)$.
\end{proof}

We also have the following applications of the first part of \fullref{lemma:quasiconvex}:
\begin{corollary}
  A geodesic triangle in which two of the sides are
  $\mu$--Morse is $\delta$--thin, with $\delta$ depending only
  on $\mu$.
\end{corollary}

\begin{corollary}\label{corollary:uniformcontractionimplieshyperbolic}
  Suppose $X$ is a geodesic metric space and $\mathcal{P}$ is a
  collection of $(\rho_1,\rho_2)$--contracting paths such that for every pair
  of points $x,\,y\in X$ there exists a 
  $\gamma\in\mathcal{P}$ with endpoints $x$ and $y$.
Then $X$ is $\delta$--hyperbolic, with $\delta$ depending only
on $\rho_1$ and $\rho_2$.
\end{corollary}
\fullref{corollary:uniformcontractionimplieshyperbolic} is an analogue
of \cite[Theorem~2.3]{MasMin99}, which is roughly the same statement
when the paths in $\mathcal{P}$ are all semi-strongly contracting with
uniform
contraction parameters.

\begin{corollary}\label{corollary:groupuniformcontractingimplieshyperbolic}
  Let $G$ be a group generated by a finite set $\mathcal{S}$.
Suppose there exist functions $\rho_1$ and $\rho_2$ and, for each
$g\in G$, a path
$\alpha_g$ from 1 to $g$ in $\Cay(G,\mathcal{S})$ that is
$(\rho_1,\rho_2)$--contracting. 
Then $G$ is hyperbolic. 
\end{corollary}

We must assume uniform contraction in
\fullref{corollary:groupuniformcontractingimplieshyperbolic}, even for
finitely presented groups. 
Dru{\c{t}}u, Mozes, and Sapir \cite{DruMozSap10} show that if $H$ is a
finitely generated subgroup of a finitely generated group $G$ and
$h\in H$ is a Morse element in $G$, that is, $\langle h\rangle$ is
Morse in some, hence, every, Cayley graph of $G$, then $h$ is a Morse element in $H$.
Thus, if $H$ is a finitely generated subgroup of a torsion-free
hyperbolic group then every element of $H$ is Morse.
However, Brady \cite{Bra99} constructed an example of a finitely presented
subgroup $H$ of a torsion-free hyperbolic group $G$ such that $H$ is
not hyperbolic.

Fink \cite{Fin15} claims that if all geodesics in a homogeneous proper
geodesic metric space are Morse, then the space is hyperbolic.
First is an assertion, \cite[Proposition~3.2]{Fin15}, that if
every geodesic is Morse then the collection of geodesics is uniformly
Morse, ie, there exists a $\mu$ such that every geodesic is $\mu$--Morse.
Then an asymptotic cone argument is used to conclude the space is
hyperbolic.
This second step can now be accomplished via our
\fullref{corollary:uniformcontractionimplieshyperbolic} without resort
to the asymptotic cone machinery.



\bibliographystyle{hypershort}
\bibliography{MDC}

\end{document}